\documentclass[11pt,a4paper]{article}

\usepackage{titlesec}
\usepackage{a4wide}
\usepackage{graphicx}
\usepackage{float}
\usepackage{amssymb}
\usepackage{amsmath}
\usepackage{amsthm}
\usepackage{appendix}
\usepackage{color}
\usepackage{array}
\usepackage{eucal}
\usepackage{mathrsfs}
\usepackage{esint}
\usepackage{lmodern}
\usepackage{tikz}
\usepackage{bbm}
\usepackage{hyperref}
\usepackage{geometry}
\usepackage{changepage}
\changepage{0pt}{}{}{}{}{0pt}{}{0pt}{10pt}
\usepackage[numbers]{natbib}
\hypersetup{
	pdfpagemode=UseThumbs,
	pdftoolbar=true,        
	pdfmenubar=true,        
	pdffitwindow=false,     
	pdfstartview={Fit},    
	pdftitle={Spectrum of the Laplacian with mixed boundary conditions in a chamfered quarter of layer},    
	pdfauthor={L. Chesnel, S.A. Nazarov, J. Taskinen},     
	pdfsubject={},  
	pdfcreator={L. Chesnel, S.A. Nazarov, J. Taskinen},   
	pdfproducer={L. Chesnel, S.A. Nazarov, J. Taskinen}, 
	pdfkeywords={}, 
	pdfnewwindow=true,      
colorlinks=true,       
linkcolor=magenta,          
citecolor=red,        
filecolor=cyan,      
urlcolor=blue           
}

\newcommand{\mrm}[1]{\mathrm{#1}}
\newcommand{\mL}{\mrm{L}}
\newcommand{\mH}{\mrm{H}}
\newcommand{\eps}{\varepsilon}
\newcommand{\dsp}{\displaystyle}
\newcommand{\Om}{\Omega}
\theoremstyle{plain}
\newtheorem{theorem}{\sffamily Theorem}[section]
\newtheorem{proposition}[theorem]{\sffamily Proposition}
\newtheorem{lemma}[theorem]{\sffamily Lemma}

\newtheorem{remark}[theorem]{\sffamily Remark}

\def\BET{\begin{theorem}}
\def\ENT{\end{theorem}}
\def\BEP{\begin{proposition}}
\def\ENP{\end{proposition}}
\def\BEL{\begin{lemma}}
\def\ENL{\end{lemma}}
\def\BER{\begin{remark} \rm}
\def\ENR{\end{remark}}

\def\bea{\begin{eqnarray}}
\def\eea{\end{eqnarray}}
\def\beas{\begin{eqnarray*}}
\def\eeas{\end{eqnarray*}}
\def\beq{\begin{equation}}
\def\eeq{\end{equation}}
\def\beal{\begin{align*}}
\def\eeal{ \end{align*} }

\def\rowleq{\nonumber \\  & \leq & }
\def\bbC{{\mathbb C}}
\def\bbN{{\mathbb N}}
\def\bbR{{\mathbb R}}
\def\ef{\eqref}

\begin{document}

~\vspace{0.0cm}
\begin{center}
{\sc \bf\fontsize{19}{19}\selectfont Spectrum of the Laplacian with mixed boundary\\[6pt] conditions in a chamfered quarter of layer}

\end{center}
\begin{center}
	\textsc{Lucas Chesnel}$^1$, \textsc{Sergei A. Nazarov}$^2$, \textsc{Jari Taskinen}$^{3}$\\[16pt]
	\begin{minipage}{0.91\textwidth}
		{\small
$^1$ IDEFIX, Ensta Paris, Institut Polytechnique de Paris, 828 Boulevard des Mar\'echaux, 91762 Palaiseau, France;\\
$^2$ Institute of Problems of Mechanical Engineering RAS, V.O., Bolshoi pr., 61, St. Petersburg, 199178, Russia;\\
$^3$ Department of Mathematics and Statistics, University of Helsinki, P.O.Box 68, FI-00014 Helsinki, Finland.\\[10pt]
			E-mails:
 \texttt{lucas.chesnel@inria.fr}, \texttt{srgnazarov@yahoo.co.uk}, \texttt{jari.taskinen@helsinki.fi}.\\[-14pt]
			\begin{center}
				(\today)
			\end{center}
		}
	\end{minipage}
\end{center}
\textbf{Abstract.} We investigate the spectrum of a Laplace operator with mixed boundary conditions in an unbounded chamfered quarter of layer. This problem arises in the study of the spectrum of the Dirichlet Laplacian in thick polyhedral domains having some symmetries such as the so-called Fichera layer. The geometry we consider depends on two parameters gathered in some vector $\kappa=(\kappa_1,\kappa_2)$ which characterizes the domain at the edges. By exchanging the axes and/or modifying their orientations if necessary, it is sufficient to restrict the analysis to the cases $\kappa_1\ge0$ and $\kappa_2\in[-\kappa_1,\kappa_1]$. We identify the essential spectrum and establish different results concerning the discrete spectrum with respect to $\kappa$. In particular, we show that for a given $\kappa_1>0$, there is some $h(\kappa_1)>0$ such that discrete spectrum exists for $\kappa_2\in[-\kappa_1,0)\cup(h(\kappa_1),\kappa_1]$ whereas it is empty for $\kappa_2\in[0,h(\kappa_1)]$. The proofs rely on classical arguments of spectral theory such as the max-min principle. The main originality lies rather in the delicate use of the features of the geometry.  \\[5pt]

\section{Formulation  of the problem}  \label{sec1}

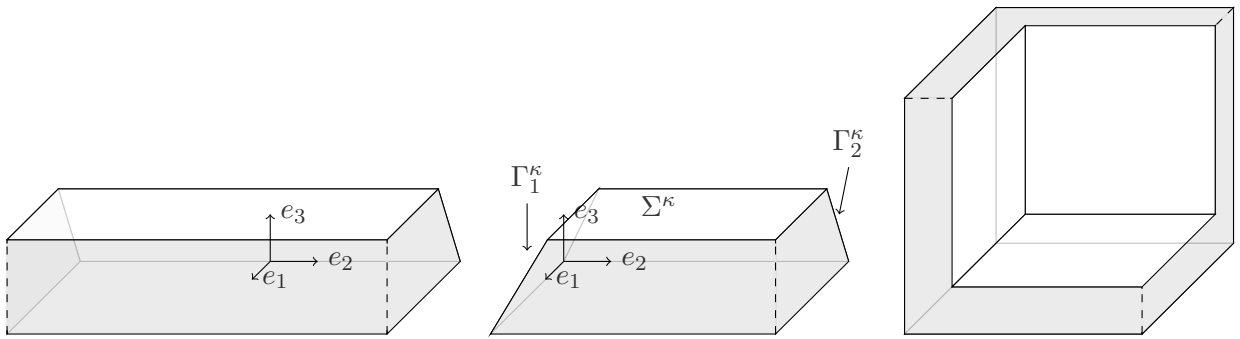
\begin{figure}[h!]
\centering
\begin{tikzpicture}[fill opacity=0.8,draw,scale=2.5]
\filldraw[fill=white] (-1,0,0) -- ++(2,0,0) -- ++(0,0.5,0.3) -- ++(-2, 0, 0) -- cycle;
\filldraw[fill=gray!20] (-1,0,0) -- ++ (0,0,1) -- ++(0,0.5,0) --++ (0,0,-0.7) -- cycle;
\filldraw[fill=gray!20] (1,0,0) -- ++ (0,0,1) -- ++(0,0.5,0) --++ (0,0,-0.7) -- cycle;
\filldraw[fill=gray!20] (-1,0,1) -- ++ (2,0,0) -- ++(0,0.5,0) --++ (-2,0,0) -- cycle;
\filldraw[fill=white] (-1,0.5,1) -- ++ (2,0,0) -- ++(0,0,-0.7) --++ (-2,0,0) -- cycle;
\draw[gray!20,line width=0.3mm] (-1,0.02,1)-- ++(0,0.46,0);
\draw[gray!20,line width=0.3mm] (1,0.02,1)-- ++(0,0.46,0);
\draw[dashed] (-1,0,1)-- ++(0,0.5,0);
\draw[dashed] (1,0,1)-- ++(0,0.5,0);
\draw[->] (0,0,0) -- (0.25,0,0) node[anchor=west]{$e_2$};
\draw[->] (0,0,0) -- (0,0.25,0) node[anchor=west]{$e_3$};
\draw[->] (0,0,0) -- (0,0,0.25) node[anchor=west]{$e_1$};
\end{tikzpicture}\quad
\begin{tikzpicture}[fill opacity=0.8,draw,scale=2.5]
\draw[<-] (1.45,0.25,0) -- (1.5,0.5,0) node[anchor=south]{$\Gamma_2^\kappa$};
\draw[<-] (0,0.25,0.5) -- (0,0.5,0.5) node[anchor=south]{$\Gamma_1^\kappa$};
\filldraw[fill=white] (0,0,0) -- ++(1.5,0,0) -- ++(0,0.5,0.3) -- ++(-1.2, 0, 0) -- cycle;
\filldraw[fill=white] (0,0,0) -- ++(0,0,1) -- ++(0.3,0.5,0) -- ++(0, 0, -0.7) -- cycle;
\filldraw[fill=gray!20] (0,0,1) -- ++(1.5,0,0) -- ++(0,0.5,0) -- ++(-1.2, 0, 0) -- cycle;
\filldraw[fill=gray!20] (1.5,0,0) -- ++(0,0,1) -- ++(0,0.5,0) -- ++(0, 0, -0.7) -- cycle;
\filldraw[fill=white] (1.5,0.5,0.3) -- ++(0,0,0.7) -- ++(-1.2,0,0) -- ++(0, 0, -0.7) -- cycle;
\draw[gray!20,line width=0.3mm] (1.5,0.02,1)-- ++(0,0.46,0);
\draw[dashed] (1.5,0,1)-- ++(0,0.5,0);
\draw[->] (0,0,0) -- (0.25,0,0) node[anchor=west]{$e_2$};
\draw[->] (0,0,0) -- (0,0.25,0) node[anchor=west]{$e_3$};
\draw[->] (0,0,0) -- (0,0,0.25) node[anchor=west]{$e_1$};
\node at (0.5,0.4,0) [anchor=north] {$\Sigma^\kappa$};
\end{tikzpicture}\quad\begin{tikzpicture}[fill opacity=0.8,draw,scale=2.5]
\filldraw[fill=white] (-0.25,-0.25,-0.25) -- ++(1.25,0,0) -- ++(0,1.25,0) -- ++(-1.25, 0, 0) -- cycle;
\filldraw[fill=white] (-0.25,-0.25,-0.25) -- ++(0,1.25,0) -- ++(0,0,1.25) -- ++(0,-1.25, 0) -- cycle;
\filldraw[fill=white] (-0.25,-0.25,-0.25) -- ++(0,0,1.25) -- ++(1.25,0,0) -- ++(0,0,-1.25) -- cycle;
\filldraw[fill=gray!20] (0,0,1) -- ++(1,0,0) -- ++(0,-0.25,0) -- ++(-1.25, 0, 0)-- ++(0,1.25,0) --++(0.25,0,0); 
\filldraw[fill=gray!20] (0,1,0) -- ++(0,0,1) -- ++(-0.25,0,0) -- ++(0, 0, -1.25)-- ++(1.25,0,0) --++(0,0,0.25); 
\filldraw[fill=gray!20] (1,0,0) -- ++(0,1,0) -- ++(0,0,-0.25) -- ++(0, -1.25,0)-- ++(0,0,1.25) --++(0,0.25,0); 
\filldraw[fill=white] (0,0,0) -- ++(1,0,0) -- ++(0,1,0) -- ++(-1, 0, 0) -- cycle;
\filldraw[fill=white] (0,0,0) -- ++(0,1,0) -- ++(0,0,1) -- ++(0,-1, 0) -- cycle;
\filldraw[fill=white] (0,0,0) -- ++(0,0,1) -- ++(1,0,0) -- ++(0,0,-1) -- cycle;
\draw[gray!20,line width=0.3mm] (-0.24,1,1)-- ++(0.23,0,0);
\draw[gray!20,line width=0.3mm] (1,-0.24,1)-- ++(0,0.23,0);
\draw[gray!20,line width=0.3mm] (1,1,-0.24)-- ++(0,0,0.23);
\draw[dashed] (-0.25,1,1)-- ++(0.25,0,0);
\draw[dashed] (1,-0.25,1)-- ++(0,0.25,0);
\draw[dashed] (1,1,-0.25)-- ++(0,0,0.25);
\end{tikzpicture}
\caption{Domains $\mathscr{B}^{\kappa_1}$ (left) and $\Omega^\kappa$ (centre). Fichera layer $\mathscr{F}$ (right).}\label{Figure3D}
\end{figure}

The original motivation for this work comes from the study of the spectrum of the Laplace operator with Dirichlet boundary conditions in thick polyhedral domains having some symmetries. The archetype of such geometries is the so-called Fichera layer 
\begin{equation}\label{defF}
\mathscr{F} := \bigcup_{j=1,2,3} \big\{ x=(x_1,x_2,x_3)\in \bbR^3\,\,|\,x_j \in (0,1),\, x_k > 0,\,k \not=j \big\}
\end{equation}
represented in Figure \ref{Figure3D} right (see \cite{Fich75} for the original article that gave rise to the name). Exploiting symmetries, in certain cases one can reduce the analysis to the one of the spectrum of the Laplace operator with mixed boundary conditions in chamfered quarters of layers. More precisely, the geometries that we consider in this article are characterized by two parameters $\kappa_1$, $\kappa_2\in\bbR$ that we gather in some vector $\kappa=(\kappa_1,\kappa_2)$. Referring to carpentry and locksmith tools, we first define the ``blade'' 
\begin{equation}\label{00}
\mathscr{B}^{\kappa_1} := \big\{ x\,|\,x_1 > \kappa_1 x_3,\,x_2 \in \bbR,\,x_3\in (0,1) \big\}
\end{equation}
(see Figure \ref{Figure3D} left). Then we introduce the  ``incisor''
\begin{equation}\label{0}
\Omega^\kappa := \{ x \in \mathscr{B}^{\kappa_1}\,|\,x_2 > \kappa_2 x_3 \}
\end{equation}
(see Figure \ref{Figure3D} centre). Let us give names to the different components of the boundary $\partial\Omega^\kappa$ of $\Omega^\kappa$. First, denote by $\Sigma^\kappa$  the union of the two ``horizontal'' quadrants:
\[
\Sigma^\kappa := \{ x\in\partial\Omega^\kappa\,|\,x_3 = 0\mbox{ or }x_3=1\}.
\]
Then consider the laterals sides of the incisor. Set
\begin{equation}\label{02}
\Gamma_1^\kappa := \partial\Omega^\kappa\cap\mathscr{B}^{\kappa_1},\qquad \Gamma_2^\kappa := \partial\Omega^\kappa\setminus\{\overline{\Gamma_1^\kappa}\cup\overline{\Sigma^\kappa}\}.
\end{equation}
We study the spectral problem with mixed boundary conditions
\begin{equation}\label{MainProblem}
\begin{array}{|rcll}
- \Delta_x u &=& \lambda u & \quad\mbox{in }\Omega^\kappa\\[3pt]
u &=& 0 & \quad\mbox{on }\Sigma^\kappa \\[3pt]
\partial_\nu u  &=& 0 & \quad\mbox{on }\Gamma_1^\kappa\cup\Gamma_2^\kappa,
\end{array}
\end{equation}
where $\partial_\nu$ is the outward normal derivative on $\partial\Omega^\kappa$. Observe that by exchanging the axes and/or modifying their orientations, there is no loss of generality to restrict the analysis to the cases
\[
\kappa_1 \ge 0,\qquad |\kappa_2| \leq \kappa_1.
\]
Denote by $\mH_0^1(\Omega^\kappa ; \Sigma^\kappa)$ the Sobolev space of functions of $\mH^1(\Omega^\kappa)$ vanishing on $\Sigma^\kappa$. Classically (see e.g. \cite{Lad,LiMa68}), the variational formulation  of Problem \ef{MainProblem} writes
\begin{equation} \label{FV}
\begin{array}{|l}
\mbox{ Find }(\lambda,u)\in\bbR\times \mH_0^1(\Omega^\kappa ; \Sigma^\kappa)\setminus\{0\}\mbox{ such that}\\[3pt]
\,(\nabla_x u , \nabla_x \psi)_{\Omega^\kappa} = \lambda\,(u ,  \psi)_{\Omega^\kappa}
\qquad\forall \psi \in \mH_0^1(\Omega^\kappa ; \Sigma^\kappa),
\end{array}
\end{equation}
where for a domain $\Xi$, $(\cdot,\cdot)_\Xi$ stands for the inner product of the Lebesgue spaces $\mL^2(\Xi)$ or $(\mL^2(\Xi))^3$ according to the case. Integrating first with respect to the $x_3$ variable and using the homogeneous Dirichlet condition on $\Sigma^\kappa$ for the functions in $\mH_0^1(\Omega^\kappa ; \Sigma^\kappa)$, one can prove that there holds the Friedrichs inequality
\[
\Vert u ; \mL^2(\Omega^\kappa) \Vert^2 \leq c_\kappa  \Vert \nabla_x u ; 
\mL^2(\Omega^\kappa) \Vert^2\qquad\forall u\in \mH_0^1(\Omega^\kappa ; \Sigma^\kappa),
\]
where $c_\kappa>0$ is a constant which depends only on $\kappa$. As known e.g. from \cite[\S10.1]{BiSo87} or \cite[Ch.\,VIII.6]{RS78}, the variational problem (\ref{FV}) gives rise to the unbounded operator $A^\kappa$ of $\mL^2(\Omega^\kappa)$ such that
\[
\begin{array}{rccl}
A^\kappa:&\mathcal{D}(A^\kappa)&\quad\to&\quad\mL^2(\Omega^\kappa)\\[6pt]
 & u & \quad\mapsto&\quad A^\kappa u=-\Delta u,
\end{array}
\]
with $\mathcal{D}(A^\kappa):=\{u\in\mH_0^1(\Omega^\kappa ; \Sigma^\kappa)\,|\,\Delta u\in\mL^2(\Omega^\kappa)\mbox{ and }\partial_\nu u=0\mbox{ on }\Gamma_1^\kappa\cup\Gamma_2^\kappa\}$. The operator $A^\kappa$ is positive definite and selfadjoint. Since $\Omega^\kappa$ is unbounded, the embedding $\mH_0^1(\Omega^\kappa ; \Sigma^\kappa)\subset \mL^2(\Omega^\kappa)$ is not compact and $A^\kappa$ has a non-empty essential component $\sigma_{\mrm{ess}}(A^\kappa)$ (\cite[Thm.\,10.1.5]{BiSo87}). Note that the case
\[
\kappa_1 = \kappa_2 = 1
\]
plays a particular role. Indeed in this situation, if $u$ is an eigenfunction associated with an eigenvalue of $A^\kappa$, by extending $u$ via even reflections with respect to the faces $\Gamma_1^\kappa$, $\Gamma_2^\kappa$, one gets an eigenvalue of the Dirichlet Laplacian in the  Fichera layer $\mathscr{F}$ defined in (\ref{defF}). This latter problem has been studied in \cite{DaLO18,BaNa21}. More precisely, in \cite{DaLO18} the authors give a characterization of the essential spectrum of the Dirichlet Laplacian and show that the discrete spectrum has at most a finite number of eigenvalues. The existence of discrete spectrum is proved in \cite[Thm.\,2]{BaNa21}.\\
\newline
The goal of this paper is to get similar information for the operator $A^\kappa$ with respect to the parameter $\kappa$. In the present work, we will also show that the spectrum of the Dirichlet Laplacian in $\Om^\kappa$, i.e. with homogeneous Dirichlet boundary conditions everywhere on $\partial\Om^\kappa$, has a rather simple structure with only essential spectrum and no discrete spectrum.\\
\newline
This note is organized as follows. In Section \ref{sec2}, we describe the essential spectrum of $A^\kappa$ (Theorem \ref{T11}). Then in Section \ref{sec2bis}, we state the results for the discrete spectrum of $A^\kappa$ (the main outcome of the present work is Theorem \ref{T24}). The next four sections contain the proof of the different items of Theorem \ref{T24}. In Section \ref{SectionNum}, we illustrate the theory with some numericals results. Finally we establish the above mentioned result related to the Dirichlet Laplacian in $\Om^\kappa$ in the Appendix (Proposition \ref{DirichletLaplacianApp}).

\section{Essential spectrum}\label{sec2}

\begin{figure}[h!]
\centering
\begin{tikzpicture}
\filldraw[fill=gray!20,draw=none] (0,0) -- ++(6,0) -- ++(0,2) -- ++(-5,0)--cycle;
\draw (6,0) -- ++(-6,0) -- ++(1,2) -- ++(5,0);
\draw[dashed] (6,0)-- ++(0,2);
\draw[->] (0,0) -- (1,0) node[anchor=north]{$\xi_1$};
\draw[->] (0,0) -- (0,1) node[anchor=east]{$\xi_2$};
\draw[<-] (0.55,1) --++ (0.6,0) node[anchor=west]{$\xi_1=\xi_2\,\tan\alpha_1$};
\begin{scope}[xshift=-2cm]
\draw[->] (0,0) -- (1,0) node[anchor=north]{$e_1$};
\draw[->] (0,0) -- (0,1) node[anchor=east]{$e_3$};
\draw[-] (0,0) -- (0,0) node[anchor=north east]{$e_2$};
\draw (0,0) circle (0.8ex);
\draw (-0.09,-0.09) -- (0.09,0.09);
\draw (0.09,-0.09) -- (-0.09,0.09);
\end{scope}
\draw[<-] (0,0) ++(90:.6) arc (90:65:.6) node[anchor=south] {\tiny $\alpha_1$}; 
\end{tikzpicture}
\caption{Domain $\Pi^{\alpha_1}$ corresponding with a cut of the blade $\mathscr{B}^{\kappa_1}$ in the plane $x_2=0$.}\label{Figure2D}
\end{figure}
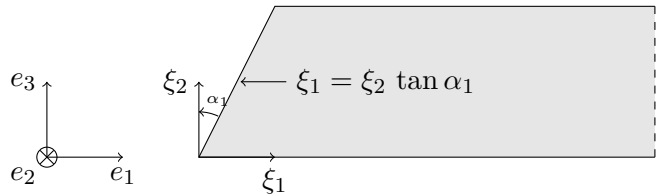

Introducing the angle $\alpha_1\in[0,\pi/2)$ such that $\kappa_1=\tan\alpha_1$, the blade \ef{00} can also be defined as 
\beas
\mathscr{B}^{\,\tan\alpha_1} = \{(x_1,x_2,x_3)\in\bbR^3\,|\, (x_1,x_3)\in\Pi^{\alpha_1}\}
\eeas
where $\Pi^{\alpha_1}$ stands for the 2D pointed strip
\begin{equation}\label{PointedStrip}
\Pi^{\alpha_1} := \big\{ \xi = (\xi_1,\xi_2) \in \bbR^2 \,|\,
\xi_1 > \xi_2\,\tan\alpha_1,\,\xi_2 \in (0,1) \big\} 
\end{equation}
(see Figure \ref{Figure2D}). To describe $\sigma_{\mrm{ess}}(A^\kappa)$,
we need information on the spectrum of the auxiliary planar problem
\begin{equation}\label{pb2D}
\begin{array}{|rcll}
- \Delta_\xi v  &=& \mu v & \quad\mbox{in }\Pi^{\alpha_1} \\[2pt]
v &=& 0 & \quad\mbox{on }\varsigma^{\alpha_1}\\[2pt]
\partial_{\nu} v&=& 0 & \quad\mbox{on }\gamma^{\alpha_1}
\end{array}
\end{equation}
where $\varsigma^{\alpha_1} := \{ \xi\in\partial\Pi^{\alpha_1} \, | \, \xi_2 = 0\mbox{ or }\xi_2 = 1 \}$ denotes the horizontal part of $\partial\Pi^{\alpha_1}$ and  $\gamma^{\alpha_1} := \partial\Pi^{\alpha_1}\setminus\overline{\varsigma^{\sigma}}$ stands for the oblique part of $\partial\Pi^{\alpha_1}$.\\
\newline
The continuous spectrum of Problem (\ref{pb2D}) coincides with the ray $[\pi^2,+\infty)$. When $\alpha_1 = 0$ (straight end), working with the decomposition in Fourier series in the vertical direction, one can prove that the discrete spectrum is empty. On the other hand, for all $\alpha_1 \in (0,\pi/2)$, it has been shown in \cite{KaNa00} that there is at least one eigenvalue below the continuous spectrum (see also
\cite{Naza12a} for more general shapes). Notice that by extending $\Pi^{\alpha_1}$ by reflection with respect to $\gamma^{\alpha_1}$, we obtain a broken strip that we can also call a V-shaped domain. This allows us to exploit all the results from \cite{ExL,na561} (see also \cite{DaRa12} as well as the amendments in \cite{Naza14c}) to get information on $\mu_1^{\alpha_1}$, the smallest eigenvalue of (\ref{pb2D}). In particular, the function $\alpha_1\mapsto \mu_1^{\alpha_1}$ is smooth and strictly decreasing on $(0,\pi/2)$. Additionally, we have 
\begin{equation}\label{defAlpha1}
\lim_{\alpha_1 \to 0^+ } \mu_1^{\alpha_1} =  \pi^2,\qquad\qquad
\lim_{\alpha_1 \to (\pi/2)^-} \mu_1^{\alpha_1} =  \frac{\pi^2}{4}
\end{equation}
(see Figure \ref{BoundEssentialSpectrum} for a numerical approximation of $\alpha_1\mapsto\mu_1^{\alpha_1}$). By adapting the approach proposed in \cite[\S3.1]{DaLO18}, one establishes the next assertion. The only point to be commented here is that there holds
\bea\label{p5}
\lambda_\dagger^\kappa := \mu_1^{\arctan\kappa_1} \leq
 \mu_1^{\arctan|\kappa_2|}  
\eea
because $|\kappa_2| \leq \kappa_1$ implies $\arctan|\kappa_2|\le \arctan\kappa_1$ and because $\alpha_1\mapsto \mu_1^{\alpha_1}$ is decreasing. 

\BET\label{T11}
The essential spectrum $\sigma_{\rm ess}(A^\kappa)$ of the operator $A^\kappa$ coincides with the ray $[\lambda_\dagger^\kappa,+\infty)$ where $\lambda_\dagger^\kappa$ is defined in \ef{p5}. 
\ENT
\begin{remark}
Thus the lower bound of $\sigma_{\rm ess}(A^\kappa)$ is characterized by the sharpest edge of $\Omega^\kappa$.
\end{remark}

\section{Discrete spectrum}\label{sec2bis}

For the discrete spectrum $\sigma_{\rm d}(A^\kappa)$, our main results are as follows:

\BET\label{T23}
For $\kappa_1=\kappa_2=0$ (straight edges), $\sigma_{\rm d}(A^\kappa)$ is empty. 
\ENT 

\BET\label{T24}
Assume that $\kappa_1 > 0$.\\[3pt]
$\textit{1)}$ $\sigma_{\rm d}(A^\kappa)$ is non-empty for $\kappa_2\in[-\kappa_1,0)$.\\
\newline
$\textit{2)}$ There exists $h(\kappa_1)\in(0,\kappa_1)$ such that:\\[3pt]
$~\hspace{2cm}i)$ $\sigma_{\rm d}(A^\kappa)$ is empty for $\kappa_2\in[0,h(\kappa_1)]$;\\[2pt]
$~\hspace{2cm}ii)$ $\sigma_{\rm d}(A^\kappa)$ is non-empty for $\kappa_2\in (h(\kappa_1),\kappa_1]$.\\
\newline
$\textit{3)}$ For $\kappa_2\in [-\kappa_1,0)\cup(h(\kappa_1),\kappa_1]$, denote by $\lambda_1^\kappa$ the first (smallest) eigenvalue of $\sigma_{\rm d}(A^\kappa)$.\\ 
The function $\kappa_2\mapsto \lambda_1^\kappa$ is strictly increasing on $[-\kappa_1,0)$ and strictly decreasing on $(h(\kappa_1),\kappa_1]$.\\
\newline
$\textit{4)}$ For $\kappa_2\in(-\kappa_1,\kappa_1)$, $\sigma_{\rm d}(A^\kappa)$ contains at most a finite number of eigenvalues. 
\ENT 

\noindent The items $\textit{1)-3)}$ of Theorem \ref{T24} are illustrated by Figure \ref{BehaviourEigen}. Note in particular that we have the following mechanism for positive $\kappa_2$: diminishing $\kappa_2$ from the value $\kappa_1$ makes the eigenvalue $\lambda_1^\kappa$ to reach the threshold $\lambda_\dagger^{\kappa}
= \mu_1^{\arctan\kappa_1}$ at a certain $\kappa_2 =  h(\kappa_1)
\in (0,\kappa_1)$. Theorem \ref{T23} can be established quite straightforwardly by working with symmetries. The rest of the present note is dedicated to the proof of the statements of Theorem \ref{T24}.

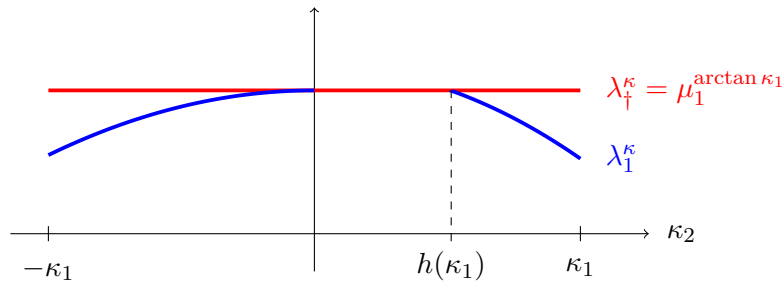
\begin{figure}[h!]
\centering
\begin{tikzpicture}
\draw[-] (-3.5,-0.1) -- (-3.5,0.1);
\draw[-] (3.5,-0.1) -- (3.5,0.1);
\draw (-3.5,-0.2) node[anchor=north]{$-\kappa_1$};
\draw (3.5,-0.2) node[anchor=north]{$\kappa_1$};
\draw (4.5,0) node[anchor=west]{$\kappa_2$};
\draw (3.7,1.9) node[anchor=west]{\textcolor{red}{$\lambda_\dagger^{\kappa}= \mu_1^{\arctan\kappa_1}$}};
\draw (3.7,1) node[anchor=west]{\textcolor{blue}{$\lambda_1^\kappa$}};
\draw[->] (-4,0) -- (4.4,0);
\draw[->] (0,-0.5) -- (0,3);
\draw (1.8,-0.1) node[anchor=north]{$h(\kappa_1)$};
\draw[domain=-3.5:3.5, smooth, variable=\x, red,line width=0.5mm] plot ({\x}, {2-0.1});
\draw[domain=-3.5:0, smooth, variable=\x, blue,line width=0.5mm] plot ({\x}, {1.9-0.07*\x*\x});
\draw[domain=1.8:3.5, smooth, variable=\x, blue,line width=0.5mm] (1.8,1.9)--(1.8,1.9) -- plot ({\x}, {2.22-0.1*\x*\x});
\draw[dashed] (1.8,-0.1) -- (1.8,1.9);
\end{tikzpicture}
\caption{Schematic picture of the behaviour of $\lambda_1^\kappa$, the smallest eigenvalue of $\sigma_{\rm d}(A^\kappa)$, for a given $\kappa_1>0$ and $\kappa_2\in[-\kappa_1,0)\cup(h(\kappa_1),\kappa_1]$.}\label{BehaviourEigen}
\end{figure}

\section{Discrete spectrum for negative $\kappa_2$}\label{sec3}

\begin{figure}[h!]
\centering

\begin{tikzpicture}[fill opacity=0.8,draw,scale=2.5]
\filldraw[fill=white] (-0.2,0,0) -- ++(0,0,1) -- ++(-0.3,0.5,0) -- ++(0, 0, -0.7) -- cycle;
\filldraw[fill=white] (-0.2,0,0) -- ++(0,0,1) -- ++(0,0.5,0) -- ++(0, 0, -0.7) -- cycle;
\filldraw[fill=gray!20] (-0.2,0,1) -- ++(0,0.5,0) -- ++(-0.3,0,0) -- cycle;
\filldraw[fill=white] (-0.2,0.5,1) -- ++(-0.3,0,0)  -- ++(0,0,-0.7) --++ (0.3,0,0) -- cycle;
\filldraw[fill=white] (0,0,0) -- ++(1.5,0,0) -- ++(0,0.5,0.3) -- ++(-1.5, 0, 0) -- cycle;
\filldraw[fill=white] (0,0,0) -- ++(0,0,1) -- ++(0,0.5,0) -- ++(0, 0, -0.7) -- cycle;
\filldraw[fill=gray!20] (0,0,1) -- ++(1.5,0,0) -- ++(0,0.5,0) -- ++(-1.5, 0, 0) -- cycle;
\filldraw[fill=gray!20] (1.5,0,0) -- ++(0,0,1) -- ++(0,0.5,0) -- ++(0, 0, -0.7) -- cycle;
\filldraw[fill=white] (1.5,0.5,0.3) -- ++(0,0,0.7) -- ++(-1.5,0,0) -- ++(0, 0, -0.7) -- cycle;
\draw[gray!20,line width=0.3mm] (1.5,0.01,1)-- ++(0,0.48,0);
\draw[dashed] (1.5,0,1)-- ++(0,0.5,0);
\draw[->] (0,0,0) -- (0.25,0,0) node[anchor=west]{$e_2$};
\draw[->] (0,0,0) -- (0,0.25,0) node[anchor=west]{$e_3$};
\draw[->] (0,0,0) -- (0,0,0.25) node[anchor=west]{$e_1$};
\node at (-0.2,0,1.2) [anchor=north] {$\Omega_-^\kappa$};
\node at (1,0,1.2) [anchor=north] {$\Omega_+^\kappa$};
\end{tikzpicture}\vspace{-0.4cm}
\caption{Domains $\Omega_-^\kappa$ and $\Omega_+^\kappa$.}\label{Figure3DDecoupage}
\end{figure}
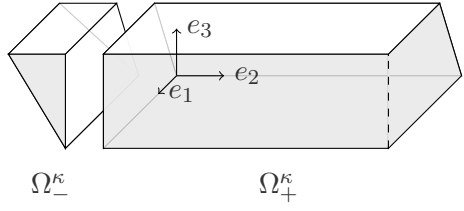

In this section, we prove the item $\textit{1)}$ of Theorem \ref{T24} and so we consider the case $\kappa_2< 0$. A direct application of the minimum principle, see e.g. \cite[Thm.\,10.2.1]{BiSo87}, \cite[Thm.\,XIII.3]{RS78}, shows that the discrete spectrum of $A^\kappa$ contains an eigenvalue $\lambda_1^\kappa$ if one can find a trial function
$\psi \in\mH_0^1(\Omega^\kappa ; \Sigma^\kappa)$ such that
\bea
\Vert \nabla_x \psi ;\mL^2(\Omega^\kappa) \Vert^2 < \lambda_\dagger^\kappa
\,\Vert \psi ;\mL^2(\Omega^\kappa) \Vert^2 . \label{11}
\eea
Let us construct a function satisfying \ef{11}. To proceed, first divide the incisor $\Omega^\kappa$ into the two domains 
\begin{equation}\label{v1}
\Omega_-^\kappa := \{ x \in \Omega^\kappa \,|\, x_2 < 0\},\qquad\quad\Omega_+^\kappa := \{ x \in \Omega^\kappa \,|\, x_2 > 0\} = \Omega^{(\kappa_1,0)}
\end{equation}
(see Figure \ref{Figure3DDecoupage}). Then for $\eps>0$ small, define $\psi^\eps$ such that
\begin{equation}\label{12}
\psi^\eps(x) = \begin{array}{|ll}
v(x_1,x_3)&\mbox{ in }\Omega_-^\kappa\\[3pt]
e^{-\eps x_2} v(x_1,x_3) &\mbox{ in }\Omega_+^\kappa
\end{array}
\end{equation}
where $v$ is an eigenfunction of the 2D problem (\ref{pb2D}) associated with $\mu_1^{\alpha_1}$, the smallest eigenvalue, and $\alpha_1 = \arctan\kappa_1$. To set ideas, we choose $v$ such that $\Vert  v;\mL^2( \Pi^{\alpha_1}) \Vert=1$. Note that $\psi^\eps$ satisfies the homogeneous Dirichlet condition on $\Sigma^\kappa$ and decays exponentially at infinity. Using (\ref{p5}), we obtain
\begin{equation}\label{Terme1}
\begin{array}{ll}
 & \Vert \nabla_x \psi^\eps;\mL^2( \Omega_+^\kappa) \Vert^2 
- \lambda_\dagger^\kappa\,\Vert \psi^\eps;\mL^2( \Omega_+^\kappa) \Vert^2 \\[6pt]
=&\dsp\big( \Vert \nabla_\xi v;\mL^2( \Pi^{\alpha_1}) \Vert^2 
+ (\eps^2 - \mu_1^{\alpha_1}) 
\Vert  v;\mL^2( \Pi^{\alpha_1}) \Vert^2  \big)\int_0^\infty e^{-2 \eps x_2} dx_2 
 = \frac{\eps}{2}\,.
\end{array}
\end{equation}
As for the integral over the prism $\Omega_-^\kappa$ with triangular cross-sections and
the bevelled end, we integrate by parts and take into account the boundary conditions of (\ref{MainProblem}), which yields
\begin{equation}\label{Terme3}
\begin{array}{ll}
\Vert \nabla_x \psi^\eps;\mL^2( \Omega_-^\kappa) \Vert^2 
- \lambda_\dagger^\kappa\,\Vert \psi^\eps;\mL^2( \Omega_-^\kappa) \Vert^2 
=&\hspace{-0.2cm}- \dsp\int_{\Omega_-^\kappa} v (x_1,x_3) 
\big( \Delta_x v (x_1,x_3)  + \mu_1^{\alpha_1}   v (x_1,x_3)\big)\,dx \\[10pt]
&+\dsp\int_{\Gamma_1^\kappa} v (x_1,x_3) \partial_{\nu} 
v (x_1,x_3)\,ds =: I_{\Omega_-^\kappa} + I_{\Gamma_1^\kappa}. 
\end{array}
\end{equation}
Owing to \ef{pb2D}, there holds $I_{\Omega_-^\kappa}=0$. Now we focus our attention on the term $I_{\Gamma_1^\kappa}$. Let $(e_1,e_2,e_3)$ denote the canonical basis of $\bbR^3$. Set $\alpha_2:=\arctan\kappa_2\in(-\pi/2,0)$ and define the new orthonormal basis $(\tilde{e}_1,\tilde{e}_2,\tilde{e}_3)$ with
\begin{equation}\label{newBasis}
\tilde{e}_1=e_1;\qquad\qquad\tilde{e}_2=\cos(\alpha_2)e_2-\sin(\alpha_2)e_3;\qquad\qquad \tilde{e}_3=\sin(\alpha_2)e_2+\cos(\alpha_2)e_3.
\end{equation}
Observe that the component $\Gamma^\kappa_1$ of the boundary of the incisor $\Om^\kappa$ is included in the plane $(O,\tilde{e}_1,\tilde{e}_3)$. 

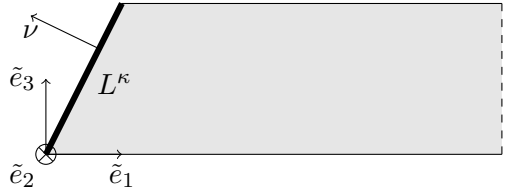
\begin{figure}[h!]
\centering
\begin{tikzpicture}
\filldraw[fill=gray!20,draw=none] (0,0) -- ++(6,0) -- ++(0,2) -- ++(-5,0)--cycle;
\draw (6,0) -- ++(-6,0) -- ++(1,2) -- ++(5,0);
\draw[dashed] (6,0)-- ++(0,2);
\draw[line width=0.8mm] (0,0)-- ++(1,2);
\node at (0.9,1.2) [anchor=north] {$L^\kappa$};

\begin{scope}[xshift=7mm,yshift=14mm]
\draw[->,rotate=154] (0,0) -- (1,0) node[anchor=north]{$\nu$};
\end{scope}
\begin{scope}[xshift=0cm]
\draw[->] (0,0) -- (1,0) node[anchor=north]{$\tilde{e}_1$};
\draw[->] (0,0) -- (0,1) node[anchor=east]{$\tilde{e}_3$};
\draw[-] (0,0) -- (0,0) node[anchor=north east]{$\tilde{e}_2$};
\draw (0,0) circle (0.8ex);
\draw (-0.09,-0.09) -- (0.09,0.09);
\draw (0.09,-0.09) -- (-0.09,0.09);
\end{scope}
\end{tikzpicture}
\caption{Domain $\Gamma_1^\kappa$ in the plane $(O,\tilde{e}_1,\tilde{e}_3)$.}\label{Figure2DGamma1}
\end{figure}

\noindent Let $(\tilde{x}_1,\tilde{x}_2,\tilde{x}_3)$ denote the coordinates in the basis (\ref{newBasis}). We have
\begin{equation}\label{Terme0}
I_{\Gamma_1^\kappa} = -\cfrac{1}{2}\,\int_{\Gamma_1^\kappa}  \frac{\partial(v^2)}{
\partial \tilde{x}_2}\,d\tilde{x}_1d\tilde{x}_3.
\end{equation}
Using that $v$ is independent of $x_2$, we obtain
\begin{equation}\label{Terme2}
0=\frac{\partial(v^2)}{
\partial x_2}=\cos\alpha_2\frac{\partial(v^2)}{
\partial \tilde{x}_2}+\sin\alpha_2\frac{\partial(v^2)}{
\partial \tilde{x}_3}\,.
\end{equation}
Combining (\ref{Terme2}) and (\ref{Terme0}), we find
\begin{equation}\label{CalculInt}
I_{\Gamma_1^\kappa} = \cfrac{\tan\alpha_2}{2}\int_{\Gamma_1^\kappa}  \frac{\partial(v^2)}{\partial \tilde{x}_3}\,d\tilde{x}_1d\tilde{x}_3=\cfrac{\tan\alpha_2}{2}\int_{\partial\Gamma_1^\kappa} v^2\,\nu\cdot\tilde{e}_3\,d\ell=\cfrac{\tan\alpha_2}{2}\int_{L^\kappa} v^2\,\nu\cdot\tilde{e}_3\,d\ell,
\end{equation}
where $L^\kappa := \{ x\in\bbR^3 \,|\, x_j = \kappa_j x_3, \ j= 1,2, \ x_3 \in (0,1) \}$ and where $\nu$ stands for the outward unit normal vector to $\partial\Gamma_1^\kappa$ (in the plane $(O,\tilde{e}_1,\tilde{e}_3)$). Using that $\alpha_2\in(-\pi/2,0)$, $\nu\cdot\tilde{e}_3>0$ on $L^\kappa$ (see Figure \ref{Figure2DGamma1}) and $v\not\equiv0$ on $L^\kappa$, we deduce that $I_{\Gamma_1^\kappa}<0$. Note also that the quantity $I_{\Gamma_1^\kappa}$ is independent of $\eps$. Gathering (\ref{Terme1}) and (\ref{Terme3}), we infer that the inequality \ef{11} holds for $\eps>0$ small enough. This is enough to guarantee that $\sigma_{\rm d}(A^\kappa)$ is non-empty for negative $\kappa_2$.

\section{Absence of eigenvalues for small positive $\kappa_2$} \label{sec4}

The goal of this section is to prove an intermediate result to establish the item $\textit{2)}$ of Theorem \ref{T24}. Therefore we assume that $\kappa_2 \ge 0$. In that situation, the integral $I_{\Gamma_1^\kappa}$ in (\ref{CalculInt}) is positive because $\alpha_2\in[0,\pi/2)$ and our argument of the previous section does not work for showing the existence of discrete spectrum. Of course, this does not yet guarantee that $\sigma_{\rm d}(A^\kappa)$ is empty. Actually we will see in Section \ref{sec5} that $\sigma_{\rm d}(A^\kappa)$
is non-empty for certain $\kappa$ with $\kappa_2> 0$. For the moment, combining the calculations of Section \ref{sec3} with the approach of \cite{na457}, we show the following result.

\BEP \label{P23}
For all $\kappa_1> 0$, there exists $\delta (\kappa_1) > 0$ such that $\sigma_{\rm d}(A^\kappa)$ is empty for 
\bea \label{15}
\kappa_2 \in[0,\delta(\kappa_1)).
\eea
\ENP

\begin{proof}
~\vspace{-0.5cm}
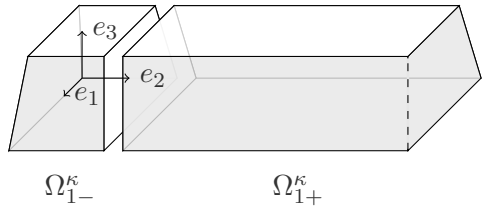
\begin{figure}[h!]
\centering
\begin{tikzpicture}[fill opacity=0.8,draw,scale=2.5]
\filldraw[fill=white] (-0.6,0,0) -- ++(0.5,0,0) -- ++(0,0.5,0.3) -- ++(-0.4,0,0) -- cycle;
\filldraw[fill=white] (-0.6,0,0) -- ++(0,0,1) -- ++(0.1,0.5,0) -- ++(0, 0, -0.7) -- cycle;
\filldraw[fill=white] (-0.1,0,0) -- ++(0,0,1) -- ++(0,0.5,0) -- ++(0, 0, -0.7) -- cycle;
\filldraw[fill=gray!20] (-0.6,0,1) -- ++(0.5,0,0) -- ++(0,0.5,0) -- ++(-0.4,0,0) -- cycle;
\filldraw[fill=white] (-0.1,0.5,1) -- ++(-0.4,0,0)  -- ++(0,0,-0.7) --++ (0.4,0,0) -- cycle;
\filldraw[fill=white] (0,0,0) -- ++(1.5,0,0) -- ++(0,0.5,0.3) -- ++(-1.5, 0, 0) -- cycle;
\filldraw[fill=white] (0,0,0) -- ++(0,0,1) -- ++(0,0.5,0) -- ++(0, 0, -0.7) -- cycle;
\filldraw[fill=gray!20] (0,0,1) -- ++(1.5,0,0) -- ++(0,0.5,0) -- ++(-1.5, 0, 0) -- cycle;
\filldraw[fill=gray!20] (1.5,0,0) -- ++(0,0,1) -- ++(0,0.5,0) -- ++(0, 0, -0.7) -- cycle;
\filldraw[fill=white] (1.5,0.5,0.3) -- ++(0,0,0.7) -- ++(-1.5,0,0) -- ++(0, 0, -0.7) -- cycle;
\draw[gray!20,line width=0.3mm] (1.5,0.01,1)-- ++(0,0.48,0);
\draw[dashed] (1.5,0,1)-- ++(0,0.5,0);
\draw[->] (-0.6,0,0) --++ (0.25,0,0) node[anchor=west]{$e_2$};
\draw[->] (-0.6,0,0) --++ (0,0.25,0) node[anchor=west]{$e_3$};
\draw[->] (-0.6,0,0) --++ (0,0,0.25) node[anchor=west]{$e_1$};
\node at (-0.2,0,1.2) [anchor=north] {$\Omega_{1-}^\kappa$};
\node at (1,0,1.2) [anchor=north] {$\Omega_{1+}^\kappa$};
\end{tikzpicture}\vspace{-0.4cm}
\caption{Domains $\Omega_{1-}^\kappa$ and $\Omega_{1+}^\kappa$.}\label{Figure3DDecoupage1}
\end{figure}

\noindent Fix $\kappa_2\in[0,\min(1,\kappa_1))$ and divide $\Omega^\kappa$ into the two domains
\bea\label{Om10}
\Omega_{1-}^\kappa := \{ x \in \Omega^\kappa \,|\, x_2< 1\}\qquad\mbox{ and }\qquad\Omega_{1+}^\kappa := \{ x \in \Omega^{\kappa} \,|\, x_2 > 1\}
\eea
(see Figure \ref{Figure3DDecoupage1}). Since $\Omega_{1+}^\kappa=\{(x_1,x_2,x_3)\in\bbR^3\,|\, (x_1,x_3)\in\Pi^{\arctan\kappa_1},\,x_2>1\}$, there holds
\bea\label{16}
\Vert \nabla_x \psi ;\mL^2(\Omega_{1+}^\kappa) \Vert^2 \geq \lambda_\dagger^\kappa\,\Vert  \psi ;\mL^2(\Omega_{1+}^\kappa) \Vert^2  
\eea
for all  $\psi \in\mH_0^1(\Omega^\kappa; \Sigma^\kappa)$. Below we show that there is some $\varrho > 0$ such that for $\kappa_2$ small enough, there holds, for all $\psi \in\mH_0^1(\Omega^\kappa; \Sigma^\kappa)$, 
\bea\label{17}
\Vert \nabla_x \psi ;\mL^2(\Omega_{1-}^\kappa) \Vert^2 \geq (\lambda_\dagger^\kappa + \kappa_2 \varrho)
\Vert  \psi ;\mL^2(\Omega_{1-}^\kappa) \Vert^2. 
\eea
Combining (\ref{16}) and (\ref{17}) with the minimum
principle yields the result of the proposition. 

\BER  
Note that estimate \ef{16} implies that $\sigma_{\rm d}(A^\kappa)$
is empty for all $\kappa=(\kappa_1,0)$ with $\kappa_1\ge0$. 
\ENR
\noindent In the remaining part of the proof, we establish \ef{17}. Consider the mixed boundary-value problem
\begin{equation}\label{18} 
\begin{array}{|rcll}
- \Delta_x w  &=& \tau w & \mbox{ in }\Omega_{1-}^\kappa\\
w  &=& 0 &\mbox{ on }\{x\in\partial \Omega_{1-}^\kappa\,|\,x_3= 0\mbox{ or }x_3=1\} 
 \\
\partial_{\nu} w  &=& 0 & \mbox{ on } \{x\in\partial \Omega_{1-}^\kappa\,|\,\ x_3\in (0,1)\}. 
\end{array}
\end{equation}
As pictured in Figure \ref{Figure3DDecoupage1} left, the domain $\Omega_{1-}^\kappa$ in \ef{Om10} is a semi-infinite prism with a trapezoidal cross-section and a skewed end. When $\kappa_2=0$, the trapezoid is simply the unit square and the
continuous spectrum of the problem \ef{18} coincides with the ray $[\pi^2,+\infty)$. In that situation, the problem (\ref{18}) admits an eigenvalue at $\mu_1^{\alpha_1} \in (0,\pi^2)$ with $\alpha_1 = \arctan\kappa_1$ (see the text above (\ref{defAlpha1}) for the definition of that quantity), a corresponding eigenfunction being $w$ such that 
\[
w(x)=v(x_1,x_3),
\]
where $v$ is an eigenfunction of (\ref{pb2D}) associated with $\mu_1^{\alpha_1}$. Now let us consider the situation $\kappa_2>0$ small. Then the map
\bea
\Omega_{1-}^\kappa \ni x \mapsto \bigg(x_1,\cfrac{x_2 - \kappa_2 x_3}{1-\kappa_2 x_3},x_3 \bigg) \in \Omega_{1-}^{(\kappa_1,0)}
\label{0x}
\eea
is a diffeomorphism whose Jacobian matrix is close to the identity and whose Hessian matrix is small. Using these properties, we deduce that the discrete spectrum of the problem \ef{18} is still non-empty for $\kappa_2$ small enough. This comes from the fact that the cut-off point of the essential spectrum satisfies the estimate
\beas
| \tau_\dagger^{\kappa_2} - \pi^2 | \leq C_\dagger \kappa_2
\eeas
and the first (smallest) eigenvalue of the discrete spectrum, which is simple\footnote{Since $\Omega_{1-}^\kappa$ is unbounded, we cannot directly apply the classical Krein-Rutman theorem to prove that the first eigenvalue $\tau_1^{\kappa_2}$ of (\ref{18}) is simple. However, this can be established for example by exploiting that the eigenfunctions associated with $\tau_1^{\kappa_2}$ are exponentially decaying at infinity and by approximating them by eigenfunctions of operators set in bounded domains (where we can apply the Krein-Rutman theorem). With this, we show that each eigenfunction associated with $\tau_1^{\kappa_2}$ is either non-negative or non-positive in $\Omega_{1-}^\kappa$, which is possible only if $\tau_1^{\kappa_2}$ is a simple eigenvalue. In the proof, one needs to use the fact that eigenfunctions cannot vanish on sets of positive area owing to the theorem on unique continuation.}, admits the expansion 
\bea\label{22}
\tau_1^{\kappa_2} = \mu_1^{\alpha_1} + \kappa_2 \tau_1' + \tilde{\tau}_1^{\kappa_2}    
\eea
with $|\tilde{\tau}_1^{\kappa_2}| \leq C\,\kappa_2^2$. Here $C>0$ is a constant independent of $\kappa_2$. These properties can be justified using classical results of the perturbation theory for linear operators, see e.g. \cite[Ch.\,7]{Kato95}, \cite[Ch.\,10]{BiSo87},
\cite[Ch.\,XII]{RS78}. From the minimum principle, to establish (\ref{17}), we see that it suffices  to show that
\bea\label{DerivPos}
\tau_1' = \frac{d \tau_1^{\kappa_2}}{d \kappa_2}\Big|_{\kappa_2=0}>0.
\eea
Let $w_1^{\kappa_2}$ be an eigenfunction of Problem (\ref{18}) associated with $\tau_1^{\kappa_2}$. Together with (\ref{22}), consider the asymptotic ansatz
\begin{equation}\label{ExpansionPos}
w_1^{\kappa_2} (x) = v(x_1,x_3) + \kappa_2 w_1'(x) + \tilde{w}_1^{\kappa_2} (x)
\end{equation}
where $\tilde{w}_1^{\kappa_2}$ is a small remainder. Insert (\ref{22}), (\ref{ExpansionPos}) into \ef{18} and collect the terms of order $\kappa_2$. We obtain
\bea\label{22}
\begin{array}{|rcll}
- \Delta_x w_1' - \mu_1^{\alpha_1} w_1' &=& \tau_1' v& \mbox{ in }\Omega_{1-}^{(\kappa_1,0)}\\
w_1' &=& 0 & \mbox{ on }\{x\in\partial \Omega_{1-}^\kappa\,|\,x_3= 0\mbox{ or }x_3=1\} . 
\end{array}
\eea
As for the Neumann boundary condition of (\ref{18}), using in particular that on $\Gamma_1^\kappa$, 
\[
\partial_{\nu}\cdot = (1 + \kappa_2^2)^{-1/2}\, \Big( - \frac{\partial\cdot}{
\partial x_2} + \kappa_2 \frac{\partial\cdot}{\partial x_3}\,\Big),
\]
at order $\kappa_2$, we find 
\bea
- \frac{\partial w_1'}{\partial x_2	}(x_1,0,x_3) = 
- \frac{\partial v}{\partial x_3	}(x_1,x_3), \qquad\quad
\frac{\partial w_1'}{\partial x_2	}(x_1,1,x_3) = 0,\qquad\quad(x_1,x_3) \in \Pi^{\alpha_1}.
\label{26}
\eea
Since the smallest eigenvalue $\mu_1^{\alpha_1}$ is simple, there exists only one compatibility condition to satisfy to ensure that the problem \ef{22}--\ef{26} has a non trivial solution. It can be written as
\[
\begin{array}{ll}
& \tau_1' = \tau_1' \Vert v ;\mL^2(\Pi^{\alpha_1})\Vert^2  \\[8pt]
=&\dsp - \int_{\Omega_{1-}^{(\kappa_1,0)}} v ( \Delta_x w_1' + \mu_1^{\alpha_1} w_1')\,dx  = \int_{\Gamma_1^{(\kappa_1,0)} } v(x_1,x_3) 
\frac{\partial w_1'}{\partial x_2} (x_1,0,x_3)\,ds \\[10pt]
=& \dsp\int_{\Pi^{\alpha_1}} v(\xi_1,\xi_2) \frac{\partial 
v}{\partial \xi_2}(\xi_1,\xi_2)\,d\xi_1d\xi_2
= \frac12 \cos \alpha_1 \int_{L^{(\kappa_1,0)}} v ^2 d \ell > 0
\end{array}
\]
where $L^{(\kappa_1,0)} := \{ (\xi_1,\xi_2)\in\bbR^2 \,|\, \xi_1 = \xi_2\tan\alpha_1,\,\xi_2 \in (0,1) \}$. This shows (\ref{DerivPos}) which guarantees that estimate (\ref{17}) is valid according to the minimum principle. Therefore the proof of Proposition \ref{P23} is complete. 
\end{proof}

\section{Existence of eigenvalues for $\kappa_2$ close to $\kappa_1>0$}
\label{sec5}

We start this section by proving that the discrete spectrum $\sigma_{\rm d}(A^\kappa)$ of the operator $A^\kappa$ can also be non-empty for certain positive $\kappa_2$. This happens
for example in the case $\kappa_1 = \kappa_2$, which we now assume. We adapt the proof of \cite[Thm.\,2]{BaNa21} and exhibit a function 
$\varphi \in\mH_0^1 (\Omega^\kappa ; \Sigma^\kappa) $ satisfying \ef{11}. First, note that for $\kappa_1 = \kappa_2$, the domain $\Omega^\kappa$ is  symmetric with respect to the ``bisector'' cross-section
\beas
\Upsilon^\kappa := \{ x \in \Omega^\kappa \,|\, x_1 =x_2\}.
\eeas
Let us divide $\Omega^\kappa$ into the two congruent domains
\bea
\Omega_\wedge^\kappa := \{ x \in \Omega^\kappa \,|\, x_1 > x_2 \}
\qquad\mbox{ and }\qquad\Omega_<^\kappa := \{ x \in \Omega^\kappa \, | \, x_2 > x_1 \}. 
\eea
Accordingly, we set  
\begin{equation}\label{14}
\psi^\eps(x) = \begin{array}{|ll}
e^{-\eps x_1} v(x_2,x_3) & \mbox{ in }\Omega_\wedge^\kappa \\[3pt]
e^{-\eps x_2} v(x_1,x_3) & \mbox{ in }\Omega_<^\kappa 
\end{array}
\end{equation}
where $v$ is as in \ef{12}. Since $\psi^\eps$ is continuous on $\Upsilon^\kappa$
and decays exponentially at infinity, it belongs to $\mH_0^1(\Omega^\kappa; \Sigma^\kappa)$. Moreover, we have
\begin{equation}\label{CalculTermeG}
\begin{array}{ll}
 &\Vert \nabla_x \psi^\eps ;\mL^2(\Omega_\wedge^\kappa) \Vert^2 - \mu_1^{\alpha_1}
\Vert \psi^\eps ;\mL^2(\Omega_\wedge^\kappa) \Vert^2 \\[10pt]
=& - \dsp\int_{\Omega_\wedge^\kappa}\hspace{-0.2cm} e^{- 2 \eps x_1} v(x_2,x_3) 
\big( \Delta_x v(x_2,x_3) + (\mu_1^{\alpha_1} + \eps^2) v(x_2,x_3)\big)\,dx \dsp\\[10pt]
 & +\dsp\int_{\Upsilon^\kappa}\hspace{-0.2cm}e^{-  \eps x_1} v(x_2,x_3) 
\partial_{\nu} \big( e^{-\eps x_1}   v(x_2,x_3) \big)\,ds =: I_{\Omega^{\kappa}}^\eps + I_{\Upsilon^{\kappa}}^\eps. 
\end{array}
\end{equation}
Using that $v$ solves \ef{pb2D}, we get $I_{\Omega^{\kappa}}^\eps = O(\eps)$. Now consider the integral $I_{\Upsilon^{\kappa}}^\eps$. Define the new orthonormal basis $(\hat{e}_1,\hat{e}_2,\hat{e}_3)$ with
\begin{equation}\label{newBasisBis}
\hat{e}_1=\cfrac{\sqrt{2}}{2}\,e_1+\cfrac{\sqrt{2}}{2}\,e_2;\qquad\qquad\hat{e}_2=-\cfrac{\sqrt{2}}{2}\,e_1+\cfrac{\sqrt{2}}{2}\,e_2;\qquad\qquad \hat{e}_3=e_3.
\end{equation}
Remark that $\Upsilon^\kappa$ is included in the plane $(O,\hat{e}_1,\hat{e}_3)$. Let $(\hat{x}_1,\hat{x}_2,\hat{x}_3)$ denote the coordinates in the basis (\ref{newBasisBis}). We have
\begin{equation}\label{Terme0Bis}
\begin{array}{rcl}
I_{\Upsilon^{\kappa}}^\eps &=& \dsp\int_{\Upsilon^\kappa}e^{-  \eps \sqrt{2}(\hat{x}_1-\hat{x}_2)/2} v\,\cfrac{\partial }{\partial\hat{x}_2} \big( e^{-\eps \sqrt{2}(\hat{x}_1-\hat{x}_2)/2} v \big)\,d\hat{x}_1d\hat{x}_3 \\[10pt]
& = & \dsp\int_{\Upsilon^\kappa}\cfrac{\eps\sqrt{2}}{2}\,e^{-  \eps \sqrt{2}(\hat{x}_1-\hat{x}_2)} v^2\,d\hat{x}_1d\hat{x}_3+\dsp\int_{\Upsilon^\kappa}\cfrac{1}{2}\,e^{-  \eps \sqrt{2}(\hat{x}_1-\hat{x}_2)}\,\cfrac{\partial(v^2) }{\partial\hat{x}_2}\,d\hat{x}_1d\hat{x}_3.
\end{array}
\end{equation}
Exploiting the exponential decay of $v (\xi)$ as $\xi_1 \to + \infty$, (see Section \ref{sec2}), one finds that the first integral of the right hand side above is $O(\eps)$. For the second one, using that $v$ is independent of $x_1$, we can write
\[
0=\frac{\partial (v^2)}{
\partial x_1}=\frac{\sqrt{2}}{2}\frac{\partial(v^2)}{
\partial \hat{x}_1}-\frac{\sqrt{2}}{2}\frac{\partial (v^2)}{
\partial \hat{x}_2}.
\]
Remarking also that $\hat{x}_2=0$ on $\Upsilon^\kappa$, this gives 
\begin{equation}\label{Terme3Bis}
\begin{array}{ll}
&\dsp\int_{\Upsilon^\kappa}\cfrac{1}{2}\,e^{-  \eps \sqrt{2}(\hat{x}_1-\hat{x}_2)}\,\cfrac{\partial(v^2) }{\partial\hat{x}_2}\,d\hat{x}_1d\hat{x}_3=\dsp\int_{\Upsilon^\kappa}\cfrac{1}{2}\,e^{-  \eps \sqrt{2}\,\hat{x}_1}\,\cfrac{\partial(v^2) }{\partial\hat{x}_1}\,d\hat{x}_1d\hat{x}_3 \\[10pt]
=&\dsp\int_{\Upsilon^\kappa}\cfrac{\eps\sqrt{2}}{2}\,e^{-  \eps \sqrt{2}\,\hat{x}_1}\,v^2\,d\hat{x}_1d\hat{x}_3+\int_{L^\kappa} \cfrac{1}{2}\,e^{-  \eps \sqrt{2}\,\hat{x}_1}\,v^2\,\nu\cdot\hat{e}_1\,d\ell
\end{array}
\end{equation}
where $L^\kappa := \{ x\in\bbR^3 \,|\, x_j = \kappa_j x_3, \ j= 1,2, \ x_3 \in (0,1) \}$ and where $\nu$ stands for the outward unit normal vector to $\Upsilon^\kappa$ (in the plane $(O,\hat{e}_1,\hat{e}_3)$). Using that $\kappa_1=\kappa_2>0$, we find  $\nu\cdot\hat{e}_1<0$ on $L^\kappa$. Since there holds $v\not\equiv0$ on $L^\kappa$, gathering (\ref{Terme0Bis}) and (\ref{Terme3Bis}), we deduce that we have $I_{\Upsilon^{\kappa}}^\eps<0$ for $\eps$ small enough. From (\ref{CalculTermeG}), we deduce
\[
\Vert \nabla_x \psi^\eps ;\mL^2(\Omega_\wedge^\kappa) \Vert^2 - \mu_1^{\alpha_1}
\Vert \psi^\eps ;\mL^2(\Omega_\wedge^\kappa) \Vert^2<0
\]
for $\eps$ small enough. Then by symmetry, we obtain
\[
\Vert \nabla_x \psi^\eps ;\mL^2(\Omega^\kappa) \Vert^2 - \mu_1^{\alpha_1}
\Vert \psi^\eps ;\mL^2(\Omega^\kappa) \Vert^2=2\Vert \nabla_x \psi^\eps ;\mL^2(\Omega_\wedge^\kappa) \Vert^2 - 2\mu_1^{\alpha_1}
\Vert \psi^\eps ;\mL^2(\Omega_\wedge^\kappa) \Vert^2<0.
\]
We conclude that the inequality \ef{11} is satisfied by the function \ef{14} which proves the following statement. 
\BET
\label{T22}
For $\kappa_1 = \kappa_2 > 0$, the discrete spectrum $\sigma_{\rm d}(A^\kappa)$ of the operator $A^\kappa$ is not empty. 
\ENT 

\noindent Since the eigenvalues of the discrete spectrum are stable with respect to small perturbations of the operator, Theorem \ref{T22} and diffeomorphisms similar to \ef{0x} imply that $\sigma_{\rm d}(A^\kappa)$ is not empty for $\kappa_2$ in a neighbourhood of $\kappa_1$. With Proposition \ref{P23}, this allows us to introduce $h(\kappa_1)\in(0,\kappa_1)$ as the infimum of the numbers $\delta$ such that $\sigma_{\rm d}(A^\kappa)$ is non-empty for all $\kappa_2\in(\delta,\kappa_1]$.\\
\newline
On the other hand, we have the following monotonicity result:
\BEL \label{P28}
Consider some $\kappa=(\kappa_1,\kappa_2)$ with $\kappa_1>0$ and $\kappa_2\in(0,\kappa_1]$ such that $A^\kappa$ has a non-empty discrete spectrum. Let $\lambda_1^\kappa$ denote the first (smallest) eigenvalue of $\sigma_{d}(A^\kappa)$. For $\eps>0$ small, set $\kappa^\eps :=
(\kappa_1, \kappa_2 + \eps)$ and denote by $\lambda_1^{\kappa^\eps}$ the first eigenvalue of $\sigma_{d}(A^{\kappa^\eps})$. Then, we have 
\bea\label{28}
\lambda_1^{\kappa^\eps}<\lambda_1^\kappa . 
\eea
\ENL
\begin{proof}
Using again the minimum principle, we can write
\begin{equation}\label{defQuotient}
\lambda_1^{\kappa^\eps} = \min_{\psi^\eps \in \mH_0^1( \Omega^{\kappa^\eps}; \Sigma^{\kappa^\eps}
) \setminus \{ 0 \}}
\frac{\big\Vert \nabla_x \psi^\eps ;\mL^2(\Omega^{\kappa^\eps}) \big\Vert^2 
}{
\big\Vert \psi^\eps;\mL^2(\Omega^{\kappa^\eps}) \big\Vert^2}\,.
\end{equation}
Now define the function $\psi^\eps$ such that
\beas
\psi^\eps(x) = u^\kappa \Big( x_1,\frac{\kappa_2 x_2}{\kappa_2 + \eps},x_3 \Big),
\eeas
where $u^\kappa \in\mH_0^1(\Omega^\kappa ; \Sigma^\kappa)$ is an eigenfunction associated with the first eigenvalue of $\sigma_{d}(A^{\kappa})$. Clearly $\psi^\eps$ is a non zero element of $\mH_0^1( \Omega^{\kappa^\eps})$. Besides, we find 
\[
\big\Vert \psi^\eps;\mL^2(\Omega^{\kappa^\eps}) \big\Vert^2=\cfrac{\kappa_2+\eps}{\kappa_2}\,\big\Vert u^\kappa;\mL^2(\Omega^{\kappa}) \big\Vert^2
\]
\[
\mbox{ and }\qquad\big\Vert \nabla_x\psi^\eps;\mL^2(\Omega^{\kappa^\eps}) \big\Vert^2=\frac{\kappa_2}{\kappa_2 + \eps} 
\,\Big\Vert \frac{\partial u^\kappa}{ \partial x_2} ;\mL^2(\Omega^\kappa ) \Big\Vert^2 
+ \cfrac{\kappa_2+\eps}{\kappa_2}\sum_{j=1,3}  \Big\Vert \frac{\partial u^\kappa}{ \partial x_j} ;\mL^2(\Omega^\kappa ) \Big\Vert^2.
\]
According to (\ref{defQuotient}), these identities imply
\beas
\lambda_1^{\kappa^\eps} &\leq& \Vert u^\kappa ;\mL^2(\Omega^\kappa ) \Vert^{-2}
\bigg( 
\frac{\kappa_2^2}{(\kappa_2 + \eps)^2} 
\,\Big\Vert \frac{\partial u^\kappa}{ \partial x_2} ;\mL^2(\Omega^\kappa ) \Big\Vert^2 
+ \sum_{j=1,3}  \Big\Vert \frac{\partial u^\kappa}{ \partial x_j} ;\mL^2(\Omega^\kappa ) \Big\Vert^2 \bigg)
\rowleq
\frac{\Vert \nabla_x u^\kappa;\mL^2(\Omega^\kappa) \Vert^2 }{\Vert u^\kappa;\mL^2(\Omega^\kappa) \Vert^2 } = \lambda_1^\kappa. 
\eeas
The strict inequality in (\ref{28}) follows from the fact that the derivative $\partial u^\kappa
/ \partial x_2$ cannot be null in the whole domain $\Omega^\kappa$. This
completes the proof of the lemma.
\end{proof}
\noindent According to relation \ef{28}, the function $\kappa_2\mapsto \lambda_1^\kappa$ is strictly decreasing on $(h(\kappa_1),\kappa_1)$. Besides, Lemma \ref{P28} ensures that $\sigma_{\rm d}(A^\kappa)$ cannot be non-empty for some $\tilde{h}(\kappa_1)\in(0,h(\kappa_1))$ otherwise  $\sigma_{\rm d}(A^\kappa)$ would be non-empty for all $\kappa_2\in(\tilde{h}(\kappa_1),\kappa_1]$ which contradicts the definition of $h(\kappa_1)$. This completes the proof of the item $\textit{2)}$ of Theorem \ref{T24}. 

\begin{remark}
For $\kappa_2\in[- \kappa_1,0)$, we have seen in Section \ref{sec3} that $\sigma_{\rm d}(A^\kappa)$ is non-empty. Let $\lambda_1^\kappa$ denote the smallest eigenvalue of $\sigma_{\rm d}(A^\kappa)$. By adapting the proof of Lemma \ref{P28}, one establishes that the map $\kappa_2\mapsto \lambda_1^\kappa$ is strictly increasing on $[- \kappa_1,0)$. Together with Lemma \ref{P28}, this shows the item $\textit{3)}$ of Theorem \ref{T24}.
\end{remark}

\section{Finiteness of the discrete spectrum}

Finally, we establish the item \textit{4)} of Theorem \ref{T24} and so assume that $\kappa_2\in(-\kappa_1,\kappa_1)$. Set again $\alpha_1=\arctan\kappa_1$, $\alpha_2=\arctan\kappa_2$. Since $|\alpha_2| <\alpha_1$, similarly to (\ref{p5}), we have 
\begin{equation}\label{FirstEstim}
\lambda_\dagger^\kappa=\mu_1^{\alpha_1} < \mu_1^{|\alpha_2|}=\mu_1^{\alpha_2}.
\end{equation}
We remind the reader that $\mu_1^{\alpha_j}$ stands for the smallest eigenvalue of the 2D problem (\ref{pb2D}) set in the pointed strip $\Pi^{\alpha_j}$ appearing in (\ref{PointedStrip}). Observe that $\Pi^{\alpha_2}$ can be obtained from $\Pi^{-\alpha_2}$ by a symmetry with respect to the line $\xi_2=1/2$ and a translation, which ensures that $\mu_1^{-\alpha_2}=\mu_1^{\alpha_2}$ and so $\mu_1^{|\alpha_2|}=\mu_1^{\alpha_2}$. 

\begin{figure}[h!]
\centering
\begin{tikzpicture}
\filldraw[fill=gray!20,draw=none] (0,0) -- ++(5,0) -- ++(0,2) -- ++(-4,0)--cycle;
\draw (5,0) -- ++(-5,0) -- ++(1,2) -- ++(4,0);
\draw (5,0)-- ++(0,2);
\draw[->] (0,0) -- (1,0) node[anchor=north]{$\xi_1$};
\draw[->] (0,0) -- (0,1) node[anchor=east]{$\xi_2$};
\draw[<-] (0.55,1) --++ (0.6,0) node[anchor=west]{$\xi_1=\xi_2\,\tan\alpha_2$};
\draw[<-] (0,0) ++(90:.6) arc (90:65:.6) node[anchor=south] {\tiny $\alpha_2$}; 
\node at (5,-0.3) [anchor=center] {$R$};
\begin{scope}[xshift=-1.8cm]
\draw[->] (0,0) -- (1,0) node[anchor=north]{$e_2$};
\draw[->] (0,0) -- (0,1) node[anchor=east]{$e_3$};
\draw[-] (0,0) -- (0,0) node[anchor=north east]{$e_1$};
\draw (0,0) circle (0.8ex);
\draw[fill=black] (0,0) circle (.3ex);
\end{scope}
\end{tikzpicture}\qquad
\begin{tikzpicture}
\filldraw[fill=gray!20,draw=none] (0,0) -- ++(6,0) -- ++(0,4) -- ++(-6,0)--cycle;
\draw (0,0)--(6,0);
\draw (0,0)--(0,4);
\draw (4,0)--(4,4);
\draw (0,1)--(4,1);
\node at (2,0.5) [anchor=center] {$\Om^\kappa_{\mrm{III}}$};
\node at (5,2) [anchor=center] {$\Om^\kappa_{\mrm{II}}$};
\node at (2,2.5) [anchor=center] {$\Om^\kappa_{\mrm{I}}$};
\node at (4,-0.3) [anchor=center] {$R$};
\begin{scope}[xshift=0cm]
\draw[->] (0,0) -- (1,0) node[anchor=north]{$e_2$};
\draw[->] (0,0) -- (0,1) node[anchor=east]{$e_1$};
\draw[-] (0,0) -- (0,0) node[anchor=north east]{$e_3$};
\draw (0,0) circle (0.8ex);
\draw (-0.09,-0.09) -- (0.09,0.09);
\draw (0.09,-0.09) -- (-0.09,0.09);
\end{scope}
\end{tikzpicture}
\caption{Left: truncated pointed strip $\Pi^{\alpha_2}(R)$. Right: bottom view of the decomposition of $\Om^\kappa$.}\label{FigureBottomView}
\end{figure}
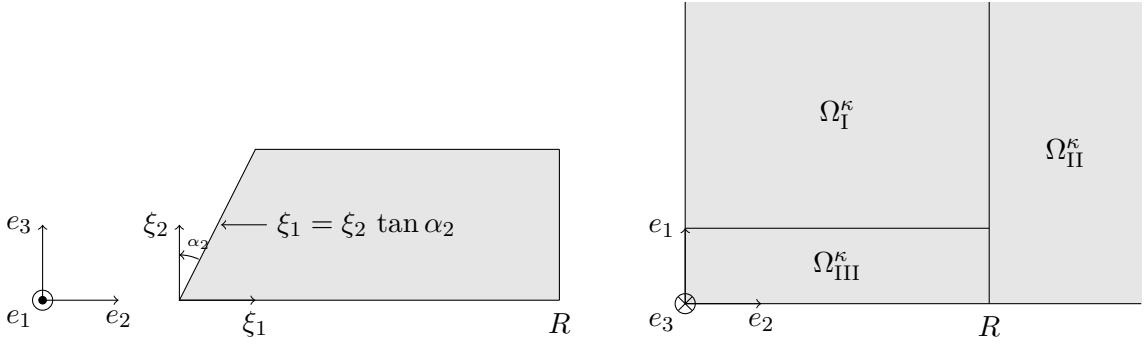

\noindent For $R>0$, define the truncated pointed strip 
\[
\Pi^{\alpha_2}(R):=\{(\xi_1,\xi_2)\in\Pi^{\alpha_2}\,|\,\xi_1<R\}
\]
(see Figure \ref{FigureBottomView} left) and consider the problem
\begin{equation}\label{pb2Dbis}
\begin{array}{|rcll}
- \Delta_\xi v  &=& \mu v & \quad\mbox{in }\Pi^{\alpha_2}(R) \\[2pt]
v &=& 0 & \quad\mbox{on }\{\xi\in\partial\Pi^{\alpha_2}(R)\,|\,\xi_2=0\mbox{ or }\xi_2=1\}\\[2pt]
\partial_{\nu} v&=& 0 & \quad\mbox{on }\{\xi\in\partial\Pi^{\alpha_2}(R)\,|\,\xi_2\ne0\mbox{ and }\xi_2\ne1\}.
\end{array}
\end{equation}

\noindent Denote by $\mu_1^{\alpha_2}(R)$ the smallest eigenvalue of (\ref{pb2Dbis}).  
Since $\mu_1^{\alpha_2}(R)$ converges to $\mu_1^{\alpha_2}$ as $R\to+\infty$, according to (\ref{FirstEstim}), we can fix $R>|\kappa_2|$ such that 
\begin{equation}\label{FirstRelation}
\mu_1^{\alpha_2}(R)>\lambda_\dagger^\kappa.
\end{equation}
Then let us divide $\Om^\kappa$ into the three domains
\[
\begin{array}{c}
\Om^\kappa_{\mrm{I}}:=\{x\in\Om^\kappa\,|\,x_1>\kappa_1\mbox{ and }x_2<R\},\qquad\qquad\Om^\kappa_{\mrm{II}}:=\{x\in\Om^\kappa\,|\,x_2>R\},\\[6pt]
\Om^\kappa_{\mrm{III}}:=\{x\in\Om^\kappa\,|\,x_1<\kappa_1\mbox{ and }x_2<R\}
\end{array}
\]
(see the representation of Figure \ref{FigureBottomView} right). Using (\ref{FirstRelation}), we obtain
\begin{equation}\label{Identi1}
\Vert \nabla_x \psi ;\mL^2(\Om^\kappa_{\mrm{I}}) \Vert^2 \geq \lambda_\dagger^\kappa\,\Vert  \psi ;\mL^2(\Om^\kappa_{\mrm{I}}) \Vert^2 \qquad\forall \psi \in \mH_0^1(\Omega^\kappa ; \Sigma^\kappa).
\end{equation}
On the other hand, from (\ref{16}), we get 
\begin{equation}\label{Identi2}
\Vert \nabla_x \psi ;\mL^2(\Om^\kappa_{\mrm{II}}) \Vert^2 \geq \lambda_\dagger^\kappa\,\Vert  \psi ;\mL^2(\Om^\kappa_{\mrm{II}}) \Vert^2 \qquad\forall \psi \in \mH_0^1(\Omega^\kappa ; \Sigma^\kappa).
\end{equation}
Besides, since $\Om^\kappa_{\mrm{III}}$ is bounded, the max-min principle (\cite[Thm. 10.2.2]{BiSo87}) guarantees that there is $n\in\bbN:=\{0,1,2,\dots\}$ such that
\begin{equation}\label{Identi3}
\lambda_\dagger^\kappa\le\max_{E\subset\mathscr{E}_n}\inf_{\psi\in E\setminus\{0\}}\cfrac{\dsp\int_{\Om^\kappa_{\mrm{III}}}|\nabla \psi|^2\,dx}{\dsp\int_{\Om^\kappa_{\mrm{III}}} \psi^2\,dx}\,,
\end{equation}
where $\mathscr{E}_n$ denotes the set of subspaces of $\mH^1_{0}(\Om^\kappa_{\mrm{III}};\Sigma_0\cap\partial\Om^\kappa):=\{\varphi\in\mH^1(\Om^\kappa_{\mrm{III}})\,|\,\varphi=0\mbox{ on }\Sigma_0\cap\partial\Om^\kappa\}$ of codimension $n$. Gathering (\ref{Identi1})--(\ref{Identi3}), we deduce that there holds 
\[
\lambda_\dagger^\kappa\le\max_{E\subset\tilde{\mathscr{E}}_n}\inf_{\psi\in E\setminus\{0\}}\cfrac{\dsp\int_{\Om^\kappa}|\nabla \psi|^2\,dx}{\dsp\int_{\Om^\kappa} \psi^2\,dx}\,,
\]
where this times $\tilde{\mathscr{E}}_n$ stands for the set of subspaces of $\mH^1_{0}(\Om^\kappa;\Sigma_0)$ of codimension $n$. From the max-min principle, this proves that $\sigma_{\rm d}(A^\kappa)$ contains at most $n$ (depending on $\kappa$) eigenvalues.

\begin{remark}
Our simple proof above does not work to show that $\sigma_{\rm d}(A^\kappa)$ is discrete when $\kappa_2=\pm\kappa_1$. However we do not expect particular phenomenon and think the result also holds in this case. It is proved in \cite[Thm.\,1.2]{DaLO18} when $\kappa_2=\kappa_1=1$.
\end{remark}

\section{Numerics and discussion}\label{SectionNum}

In this section, we illustrate some of the results above. In Figure \ref{BoundEssentialSpectrum}, we represent an approximation of the first eigenvalue of the 2D problem (\ref{pb2D}) set in the pointed strip $\Pi^{\alpha_1}$ with respect to $\alpha_1\in(0,9\pi/20)$. We use a rather crude method which consists in truncating the domain at $\xi_2=12$ (see the picture of Figure \ref{FigureBottomView} left) and imposing homogeneous Dirichlet boundary condition on the artificial boundary. Then we compute the spectrum in this bounded geometry by using a classical P2 finite element method. To proceed, we use the library \texttt{Freefem++} \cite{Hech12} and display the results with \texttt{Matlab}\footnote{\texttt{Matlab}, \url{http://www.mathworks.com/}.} and \texttt{Paraview}\footnote{\texttt{Paraview}, \url{http://www.paraview.org/}.}. The values we get are coherent with the ones recalled in (\ref{defAlpha1}). For more details concerning the numerical analysis of this problem, we refer the reader to \cite{DaLR12}. In this work, one can also find an interesting study concerning the case $\alpha\to(\pi/2)^-$.\\
\newline
In Figure \ref{Figk1m1}--\ref{Figk11}, we fix $\kappa_1=1$ (equivalently $\alpha_1=\pi/4$) and compute the first eigenvalue of $\sigma_{\rm d}(A^\kappa)$ for $\kappa_2\in\{-1,-0.1,1\}$. For $\kappa_1=1$, the bound of the essential spectrum of $A^\kappa$ is $\lambda^\kappa_\dagger\approx0.929\pi^2$ (see Figure  \ref{BoundEssentialSpectrum} as well as \cite{DaLO18}). For each of the three $\kappa_2$, in agreement with Theorem \ref{T24}, we find an eigenvalue below the essential spectrum.   Actually, in each situation our numerical experiments seem to indicate that there is only one eigenvalue in the discrete spectrum, which is a result that we have not proved. Interestingly, for $\kappa_2=-0.1$ (Figure \ref{Figk1m01}), the eigenfunction is not particularly localized at the intersection of the obliques sides. This is related to the so-called Agmon estimates which guarantee that the decay rate coincides with the square root of the difference between the lower bound of the essential spectrum $\lambda^\kappa_\dagger$ and the eigenvalue $\lambda^\kappa_1$ (see \cite{Agmon14,DaRa12}). For $\kappa_2=-0.1$ the quantity $\lambda^\kappa_\dagger-\lambda^\kappa_1$ is rather small. We emphasize that here we simply compute the spectrum of the Laplace operator with mixed boundary conditions in the bounded domain $\{x\in\Om^\kappa\,|x_1<6\mbox{ and }x_2<6\}$. At $x_1=6$ and $x_2=6$, we impose homogeneous Neumann boundary condition. Admittedly, this is a very naive approximation, especially in the case of Figure \ref{Figk1m01}. Our approximation lacks precision in that situation, this is the reason why we do not give the value of the corresponding eigenvalue.\\
\newline
In Figure \ref{Figk3m3}, we represent eigenfunctions associated with two different eigenvalues of $\sigma_{\rm d}(A^\kappa)$ for $\kappa=(3,-3)$. For the 2D problem (\ref{pb2D}) in the pointed strip, the cardinal of the discrete spectrum can be made as large as desired by considering sufficiently sharp angles. We imagine that a similar phenomenon occurs in our geometry $\Om^\kappa$. However to prove such a result is an open problem. At least the numerics of  Figure \ref{Figk3m3} suggest that we can have more than one eigenvalue in $\sigma_{\rm d}(A^\kappa)$.\\
\newline
As mentioned in the introduction, the fact that the discrete spectrum of $A^\kappa$ for $\kappa=(1,1)$ is not empty ensures that the Dirichlet Laplacian in the so-called Fichera layer $\mathscr{F}$ of Figure \ref{Figure3D} right admits an eigenvalue below the essential spectrum. This can be proved by playing with symmetries and reconstructing $\mathscr{F}$ from three versions of $\Om^{(1,1)}$. Now, gluing six domains $\Om^{(1,-1)}$, we can create the cubical structure pictured in Figure \ref{FigCubicalStructure} (note that its boundary is not Lipschitz). Then, from Theorem \ref{T24} which guarantees that $\sigma_{\rm d}(A^\kappa)$ contains at least one eigenvalue, we deduce that the Dirichlet Laplacian in this geometry has at least one eigenvalue.

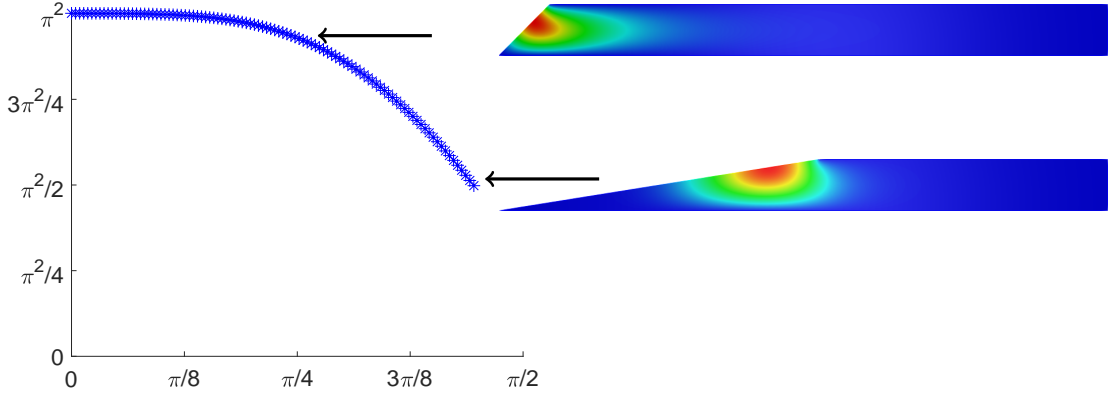
\begin{figure}[!ht]
\hspace{0.5cm}
\includegraphics[width=7cm]{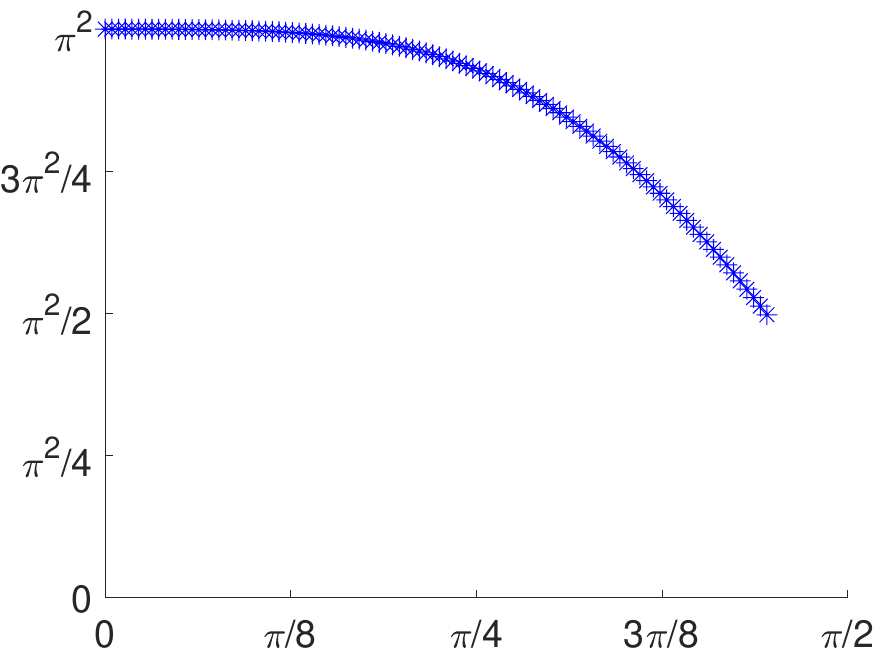}
\raisebox{3.6cm}{\hspace{-1cm}
\begin{tabular}{l}
\includegraphics[width=8cm]{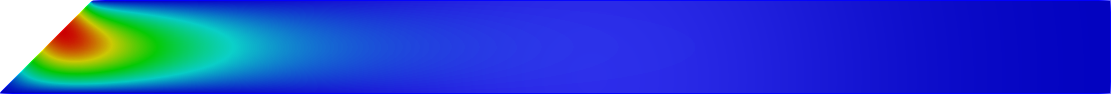}\\[35pt]
\includegraphics[width=8cm]{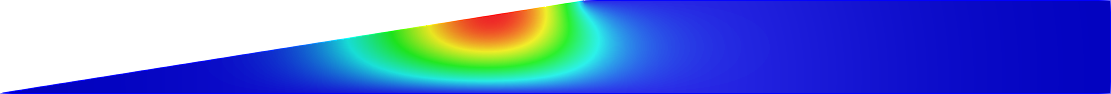}
\end{tabular}
}
\raisebox{2.8cm}{\hspace{-8.8cm}
\begin{tikzpicture}[scale=1]
\draw[black,->,very thick] (-0.5,0)--(-2,0);
\end{tikzpicture}}
\raisebox{4.7cm}{\hspace{-4cm}
\begin{tikzpicture}[scale=1]
\draw[black,->,very thick] (-0.5,0)--(-2,0);
\end{tikzpicture}}
\caption{Curve $\alpha_1\mapsto \mu^{\alpha_1}_1$ for $\alpha_1\in(0,9\pi/20)$. According to Theorem \ref{T11}, this gives the bound $\lambda^\kappa_\dagger$ of the essential spectrum of $A^\kappa$. The two pictures correspond to eigenfunctions associated with $\mu^{\alpha_1}_1$ for $\alpha_1=\pi/4$ and $\alpha_1=9\pi/20$.}
\label{BoundEssentialSpectrum}
\end{figure}

\begin{figure}[!ht]
\centering
\includegraphics[width=7.5cm]{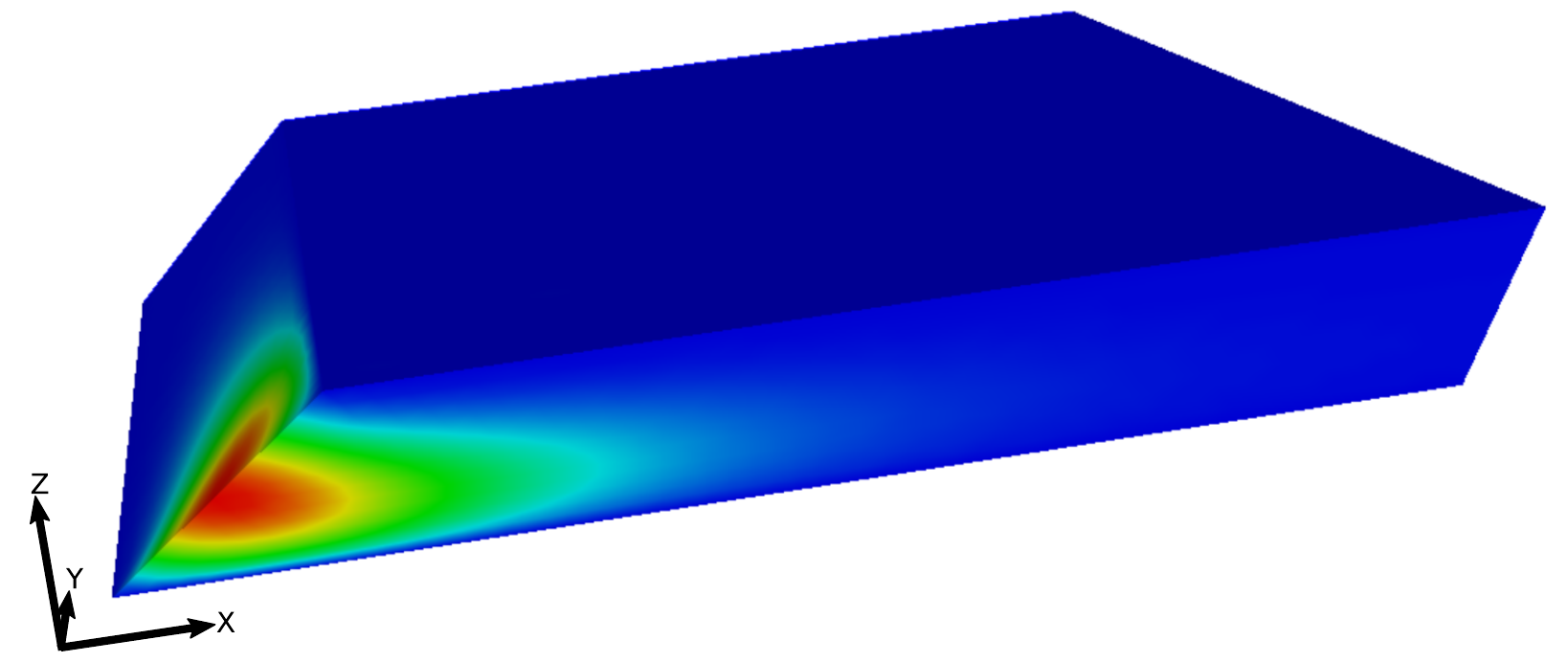}
\includegraphics[width=8.5cm]{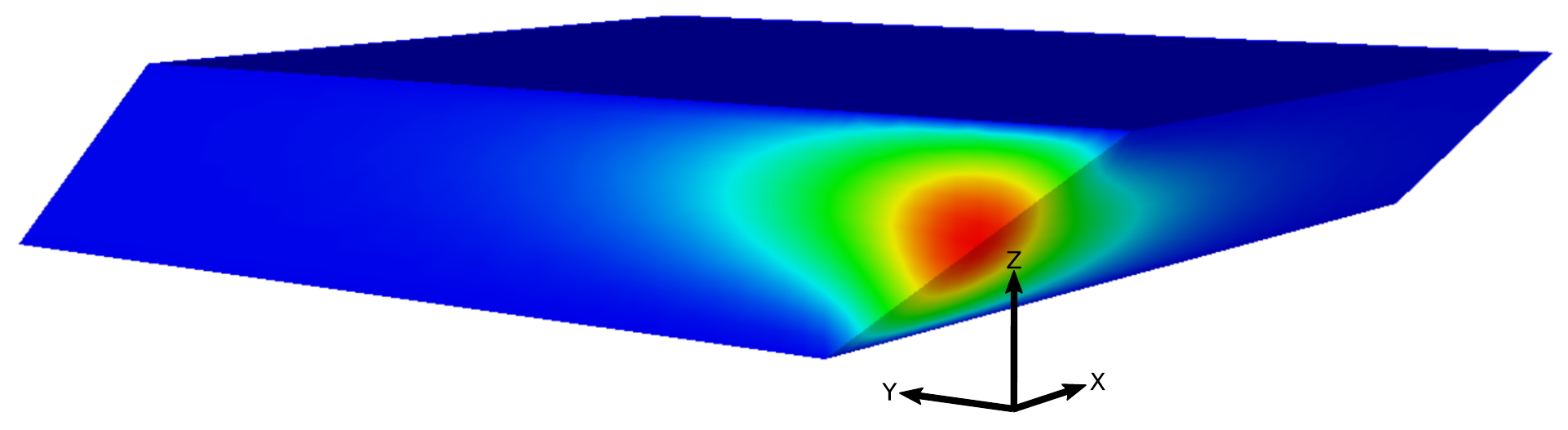}
\caption{Two views of an eigenfunction associated with the first eigenvalue of $\sigma_{\rm d}(A^\kappa)$ for $\kappa=(1,-1)$. We find $\lambda^\kappa_1\approx0.81\pi^2$.}
\label{Figk1m1}
\end{figure}

\begin{figure}[!ht]
\centering
\includegraphics[width=7cm]{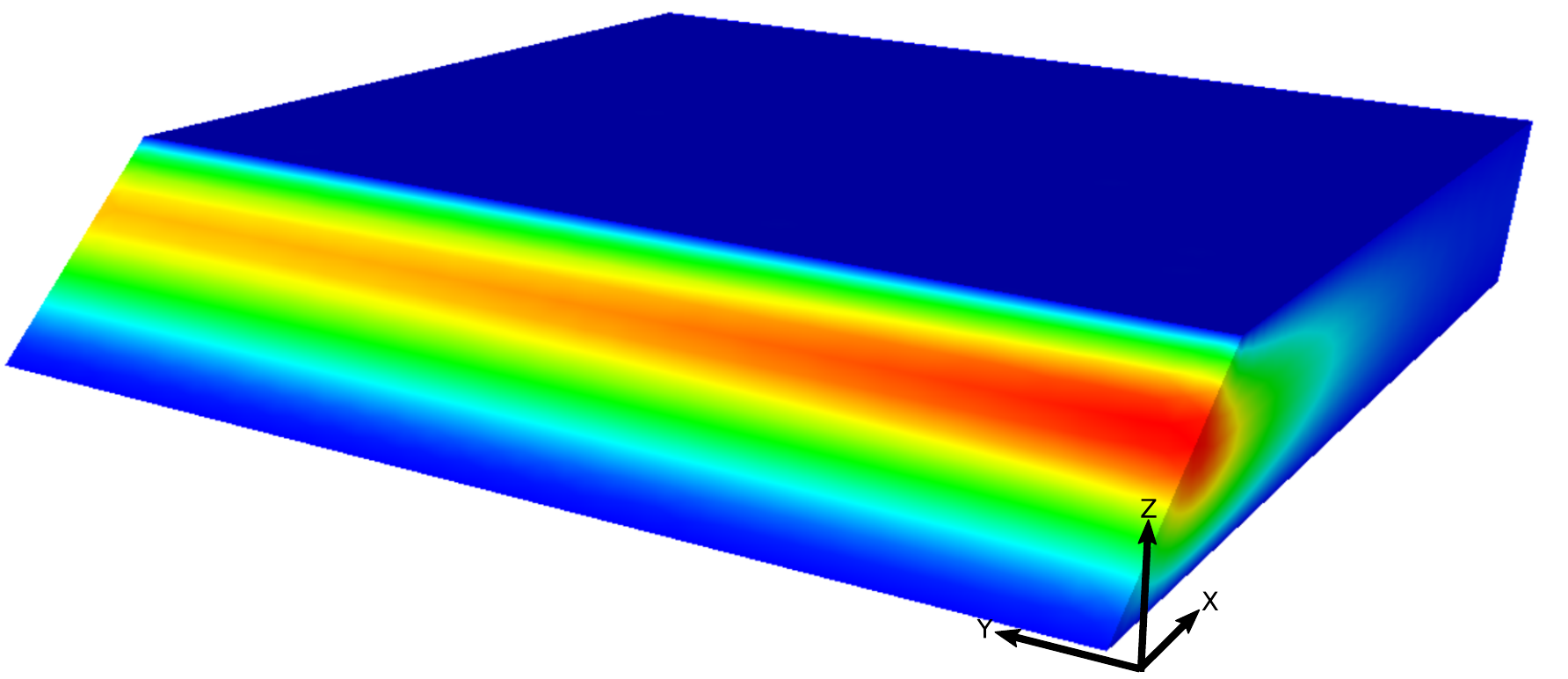}
\caption{Eigenfunction associated with the first eigenvalue of $\sigma_{\rm d}(A^\kappa)$ for $\kappa=(1,-0.1)$.}
\label{Figk1m01}
\end{figure}

\begin{figure}[!ht]
\centering
\includegraphics[width=7cm]{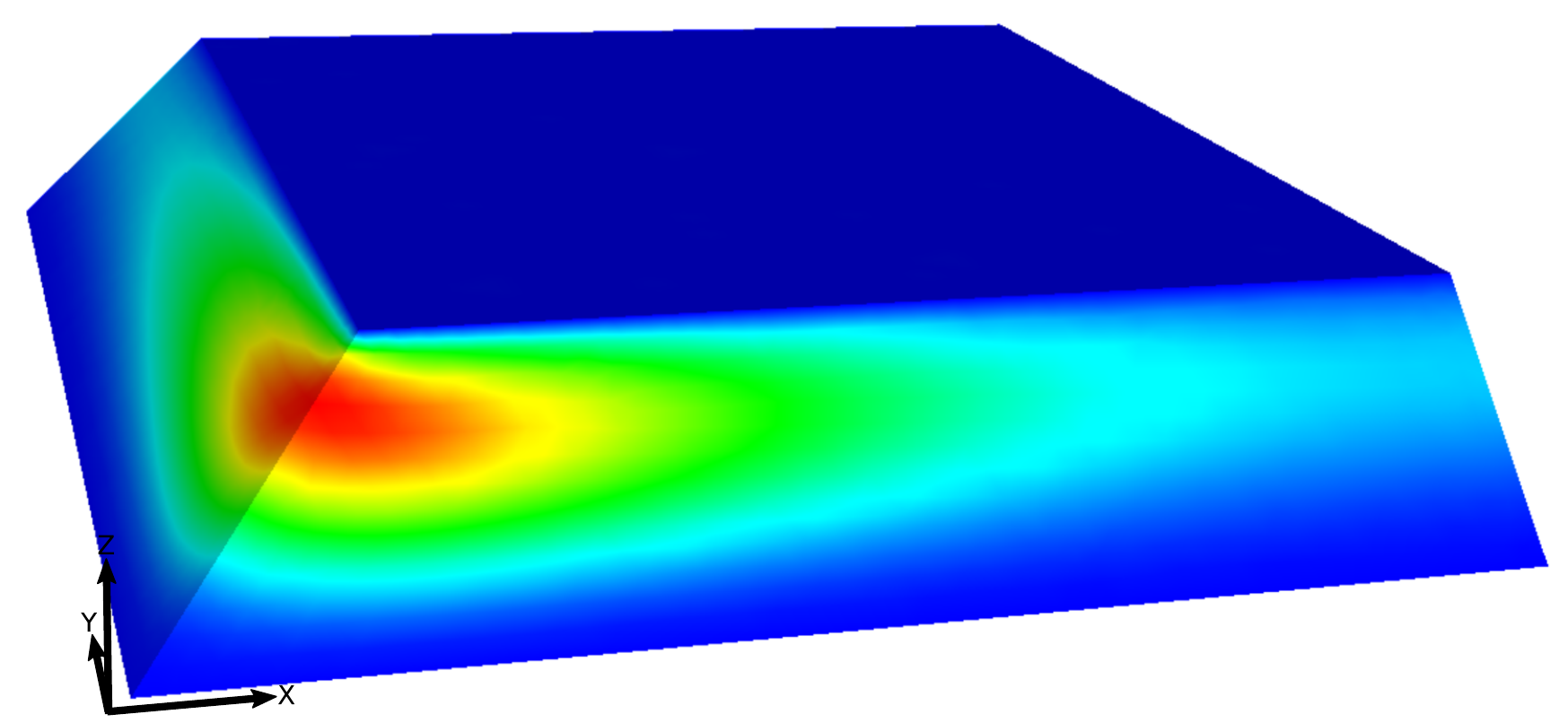}
\caption{Eigenfunction associated with the first eigenvalue of $\sigma_{\rm d}(A^\kappa)$ for $\kappa=(1,1)$. We find $\lambda^\kappa_1\approx0.90\pi^2$.}
\label{Figk11}
\end{figure}

\begin{figure}[!ht]
\centering
\includegraphics[width=7cm]{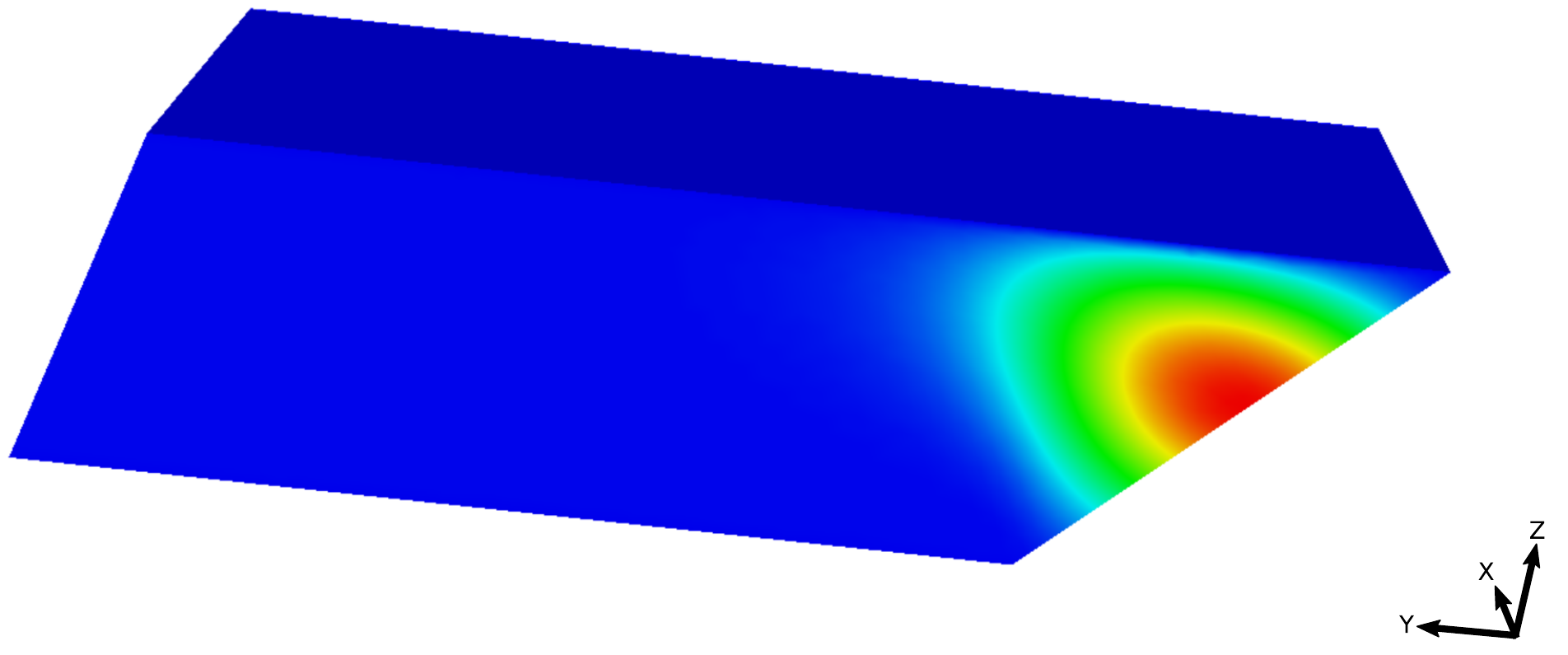}\qquad
\includegraphics[width=8cm]{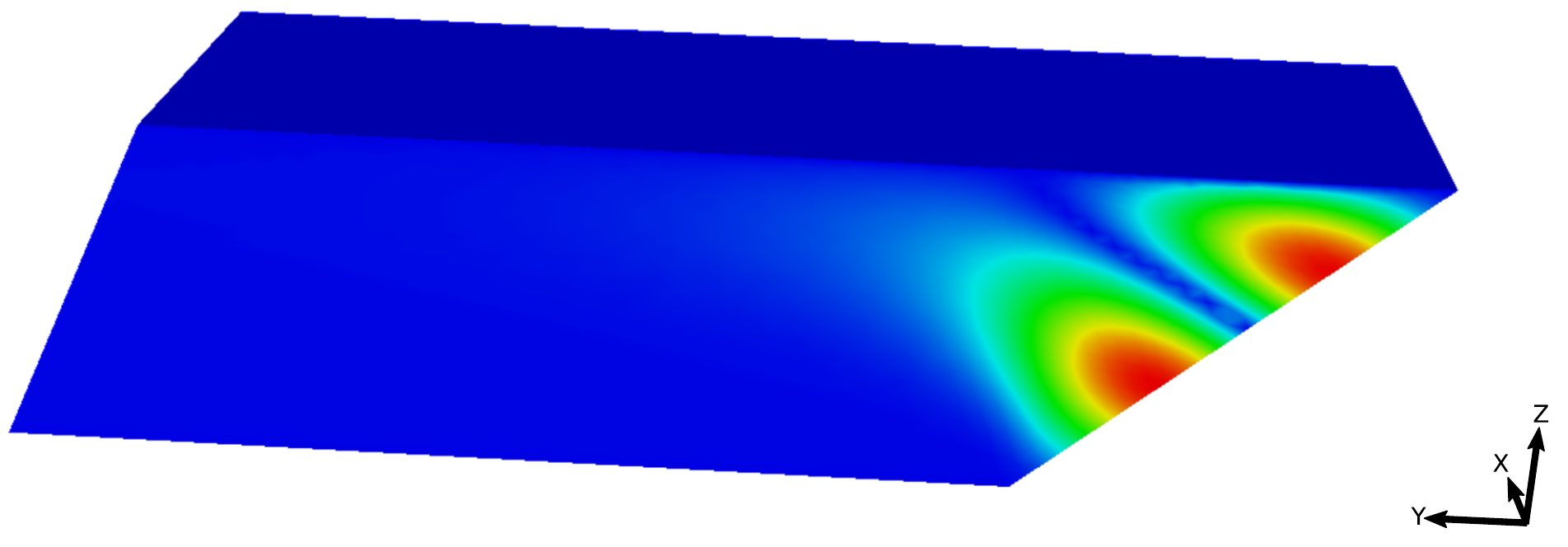}
\caption{Eigenfunctions associated with two different eigenvalues of $\sigma_{\rm d}(A^\kappa)$ for $\kappa=(3,-3)$.}
\label{Figk3m3}
\end{figure}

\begin{figure}[!ht]
\centering
\includegraphics[width=11cm]{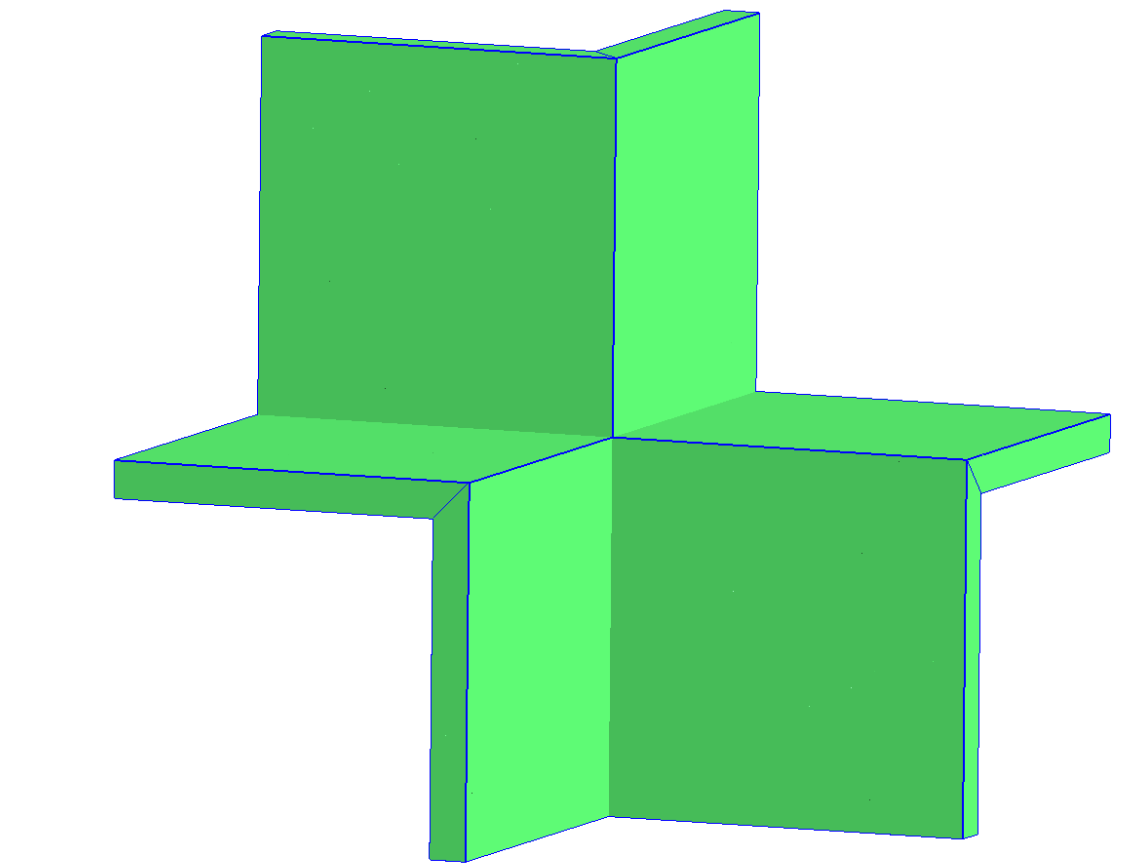}
\caption{Cubical structure obtained by gluing six domains $\Om^{\kappa}$ with $\kappa=(1,-1)$.}
\label{FigCubicalStructure}
\end{figure}

\newpage
\clearpage

\section*{Appendix}

Here we show that the Dirichlet Laplacian in $\Om^\kappa$ has no isolated eigenvalue nor eigenvalues embedded in the essential spectrum.

\begin{proposition}\label{DirichletLaplacianApp}
Fix $\kappa=(\kappa_1,\kappa_2)\in\bbR^2$. Assume that $u\in\mH^1(\Om^\kappa)$ satisfies
\begin{equation}\label{MainProblemD}
\begin{array}{|rcll}
- \Delta_x u &=& \lambda u & \quad\mbox{in }\Omega^\kappa\\[3pt]
u &=& 0 & \quad\mbox{on }\partial\Omega^\kappa
\end{array}
\end{equation}
for some $\lambda\in\bbC$. Then there holds $u\equiv0$ in $\Om^\kappa$.
\end{proposition}
\begin{proof}
The result is clear when $\lambda\in\bbC\backslash(0,+\infty)$. Let us apply the Rellich trick \cite{Rell43} to deal with the case $\lambda>0$. If $u$ solves (\ref{MainProblemD}), the function $\partial_{x_1}u$ satisfies
\begin{equation}\label{MainProblemDeriv}
\begin{array}{|rcll}
- \Delta_x(\partial_{x_1}u) &=& \lambda\,\partial_{x_1}u & \quad\mbox{in }\Omega^\kappa\\[10pt]
\partial_{x_1}u &=& 0 & \quad\mbox{on }\Sigma^\kappa.
\end{array}
\end{equation}
Note that since $\Omega^\kappa$ is convex, classical regularity results (see e.g., the second basic inequality in \cite[\S II.6]{Lad}) ensure that $u$ belongs to $\mH^2(\Om^\kappa)$ so that $\partial_{x_1}u$ falls in $\mH^1(\Om^\kappa)$. Multiplying (\ref{MainProblemD}) by $\partial_{x_1}u$, (\ref{MainProblemDeriv}) by $u$, integrating by parts and taking the difference, we obtain
\begin{equation}\label{identityIPP}
0=\int_{\Gamma_2^\kappa}\cfrac{\partial u}{\partial \nu}\,\cfrac{\partial u}{\partial x_1}\,ds
\end{equation}
(note that $\partial_{x_1}u=0$ on $\Gamma_1^\kappa$ because $u=0$ on $\Gamma_1^\kappa$). On $\Gamma_2^\kappa$, the fact that $u=0$ implies 
\[
-\kappa_1\cfrac{\partial u}{\partial x_1}=\cfrac{\partial u}{\partial x_3}\,.
\]
Therefore we get
\[
\cfrac{\partial u}{\partial \nu}=(1 + \kappa_1^2)^{-1/2}\,\,\Big(-\cfrac{\partial u}{\partial x_1}+\kappa_1\cfrac{\partial u}{\partial x_3}\,\Big)=-\sqrt{1+\kappa_1^2}\,\cfrac{\partial u}{\partial x_1}\,.
\]
Using this in (\ref{identityIPP}) gives $\partial u/\partial \nu=0$ on $\Gamma_2^\kappa$. Since we also have $u=0$ on $\Gamma_2^\kappa$ and $\Delta_xu+\lambda u=0$ in $\Om^\kappa$, the theorem of unique continuation (see e.g. \cite{Bers}, \cite[Thm.\,8.6]{CoKr13}) guarantees that $u\equiv0$ in $\Om^\kappa$.
\end{proof}

\section*{Acknowledgements}
The work of the second author was supported by the Russian Science Foundation, project 22-11-00046.

\bibliography{Bibli}

\def\cprime{$'$}
\begin{thebibliography}{10}

\bibitem{Agmon14}
S.~Agmon.
\newblock {\em Lectures on exponential decay of solutions of second-order
  elliptic equations: bounds on eigenfunctions of {N}-body {S}chr\"odinger
  operators}, volume~29 of {\em Math. Notes}.
\newblock Princeton University Press, 2014.

\bibitem{BaNa21}
{F.L.} Bakharev and {A.I.} Nazarov.
\newblock Existence of the discrete spectrum in the {Fichera} layers and
  crosses of arbitrary dimension.
\newblock {\em J. Funct. Anal.}, 281(4):109071, 2021.

\bibitem{Bers}
L.~Bers, F.~John, and M.~Schechter.
\newblock {\em Partial differential equations}.
\newblock Springer-Verlag, John Wiley, 1964.

\bibitem{BiSo87}
{M.Sh.} Birman and {M.Z.} Solomjak.
\newblock {\em Spectral theory of selfadjoint operators in {H}ilbert space}.
\newblock Mathematics and its Applications (Soviet Series). D. Reidel
  Publishing Co., Dordrecht, 1987.

\bibitem{CoKr13}
D.~Colton and R.~Kress.
\newblock {\em Inverse acoustic and electromagnetic scattering theory. 3rd
  ed.}, volume~93 of {\em Appl. Math. Sci.}
\newblock Springer-Verlag, Berlin, 2013.

\bibitem{DaLO18}
M.~Dauge, Y.~Lafranche, and T.~Ourmi{\`e}res-Bonafos.
\newblock Dirichlet spectrum of the {Fichera} layer.
\newblock {\em Integral Equ. Oper. Theory}, 90(5):60, 2018.

\bibitem{DaLR12}
M.~Dauge, Y.~Lafranche, and N.~Raymond.
\newblock Quantum waveguides with corners.
\newblock In {\em ESAIM Proc.}, volume~35, pages 14--45, 2012.

\bibitem{DaRa12}
M.~Dauge and N.~Raymond.
\newblock Plane waveguides with corners in the small angle limit.
\newblock {\em J. Math. Phys.}, 53(12):123529, 2012.

\bibitem{ExL}
P.~Exner, P.~Seba, and P.~Stovicek.
\newblock On existence of a bound state in an {L}-shaped waveguide.
\newblock {\em Czech J. Phys.}, 39:1181--1191, 1989.

\bibitem{Fich75}
G.~Fichera.
\newblock Asymptotic behaviour of the electric field and density of the
  electric charge in the neighbourhood of singular points of a conducting
  surface.
\newblock {\em Russ. Math. Surv.}, 30(3):107, 1975.

\bibitem{Hech12}
F.~Hecht.
\newblock New development in freefem++.
\newblock {\em J. Numer. Math.}, 20(3-4):251--265, 2012.
\newblock \url{http://www3.freefem.org/}.

\bibitem{KaNa00}
{I.V.} Kamotskii and {S.A.} Nazarov.
\newblock On eigenfunctions localized in a neighborhood of the lateral surface
  of a thin domain.
\newblock {\em J. Math. Sci.}, 101(2):2941--2974, 2000.

\bibitem{Kato95}
T.~Kato.
\newblock {\em {Perturbation theory for linear operators.}}
\newblock Springer-Verlag, Berlin, reprint of the corr. print. of the 2nd ed.
  1980 edition, 1995.

\bibitem{Lad}
{O.A.} Ladyzhenskaya.
\newblock {\em The Boundary Value Problems of Mathematical Physics}.
\newblock Nauka, Moscow, 1973.
\newblock (English transl.: Springer-Verlag, New York, 1985).

\bibitem{LiMa68}
J.-L. Lions and E.~Magenes.
\newblock {\em {Probl\`{e}mes} aux limites non {homog\`{e}nes} et
  applications}.
\newblock Dunod, 1968.

\bibitem{na457}
{S.A.} Nazarov.
\newblock Variational and asymptotic methods for finding eigenvalues below the
  continuous spectrum threshold.
\newblock {\em Sibirsk. Mat. Zh.}, 51:1086--1101, 2010.
\newblock (English transl. Siberian Math. J. 51, 866--878, 2010.).

\bibitem{Naza12a}
{S.A.} Nazarov.
\newblock Discrete spectrum of cranked, branching, and periodic waveguides.
\newblock {\em Algebra i analiz}, 23(2):206--247, 2011.
\newblock (English transl.: St. Petersburg Math. 23:2, 351--379, 2012.).

\bibitem{Naza14c}
{S.A.} Nazarov.
\newblock Asymptotics of eigenvalues of the {Dirichlet} problem in a skewed
  $\mathscr{T}$-shaped waveguide.
\newblock {\em Comput. Math. Math. Phys.}, 54:811--830, 2014.

\bibitem{na561}
{S.A.} Nazarov and {A.V.} Shanin.
\newblock Trapped modes in angular joints of {2D} waveguides.
\newblock {\em Appl. Anal.}, 93:572--582, 2014.

\bibitem{RS78}
M.~Reed and B.~Simon.
\newblock {\em Methods of modern mathematical physics. IV. Analysis of
  operators}.
\newblock Academic Press, New-York, 1978.

\bibitem{Rell43}
F.~Rellich.
\newblock {\"U}ber das asymptotische {Verhalten} der {L{\"o}sungen} von
  {{\(\Delta u+\lambda u=0\)}} in unendlichen {Gebieten}.
\newblock {\em Jahresber. Dtsch. Math.-Ver.}, 53:57--65, 1943.

\end{thebibliography}
\bibliographystyle{plain}
\end{document}